\documentclass{amsart}
\usepackage{amsmath}
\usepackage{amssymb}
\usepackage[all]{xy}
\usepackage{comment}
\usepackage{color}

\theoremstyle{plain}
\newtheorem{thm}{Theorem}[section]
\newtheorem{thm*}{Theorem}[section]
\newtheorem{cor}[thm]{Corollary}
\newtheorem{prop}[thm]{Proposition}
\newtheorem{lemma}[thm]{Lemma}

\newtheorem{lemma*}{Lemma}

\newtheorem{question}[thm]{Question}

\theoremstyle{definition}
\newtheorem{defn}[thm]{Definition}
\newtheorem{remark}[thm]{Remark}

\newtheorem{note}[thm]{Notation}

\newtheorem*{remark*}{Remark}

\newtheorem{ex}[thm]{Example}

\numberwithin{equation}{thm}

\newcommand{\cN}{\mathcal N}

\newcommand{\bG}{\mathbb G}

\newcommand{\bA}{\mathbb A}

\newcommand{\cP}{\mathcal P}

\newcommand{\bZ}{\mathbb Z}
\newcommand{\bF}{\mathbb F}

\newcommand{\cC}{\mathcal C}
\newcommand{\cS}{\mathcal S}

\newcommand{\cE}{\mathcal E}

\newcommand{\fe}{\mathfrak e}

\newcommand{\fg}{\mathfrak g}

\newcommand{\fu}{\mathfrak u}

\newcommand{\gl} {\mathfrak {gl}}

\newcommand{\fv}{\mathfrak v}

\newcommand{\ol}{\overline}
\newcommand{\ul}{\underline}
\def\sl2{\operatorname{SL_{2(2)}}\nolimits}
\def\Ga2{\operatorname{\mathbb G_{\rm a(2)}}\nolimits}

\newcommand{\bu}{\bullet}

\date\today

\begin{document}

 \title[Cohomology of Unipotent Group Schemes]{Cohomology of Unipotent Group Schemes}
 
 \author[ Eric M. Friedlander]
{Eric M. Friedlander$^{*}$} 

\address {Department of Mathematics, University of Southern California,
Los Angeles, CA}
%%% Email address is optional.
\email{ericmf@usc.edu}
%\email{eric@math.northwestern.edu}

\thanks{$^{*}$ partially supported by the Simons Foundation}



\begin{abstract} 
We verify that universal classes in the cohomology of $GL_N$ determine explicit 
cohomology classes of Frobenius kernels $G_{(r)}$ of various linear algebraic groups $G$ .  
We  consider the relationship of $\varprojlim_r H^*(U_{(r)},k)$ to the rational cohomology 
$H^*(U,k)$ of many unipotent algebraic groups $U$.  The second half of this
paper investigates in detail the cohomology of Frobenius kernels $(U_3)_{(r)}$
of the Heisenberg group $U_3 \subset GL_3$.
\end{abstract}

\maketitle

\tableofcontents

\section{Introduction}

We consider linear algebraic groups $G$ defined over a field of characteristic $p > 0$ and their
Frobenius kernels $G_{(r)}$.   We investigate the rational cohomology algebra $H^*(G,k)$ of $G$ and
the cohomology algebra $H^*(G_{(r)},k)$. Our results are of two types.  The
first two sections are of a general nature, applying to a wide class of unipotent groups.  The next two
sections provide more detailed information for $H^*((U_3)_{(r)},k)$, where $U_3 \subset GL_3$ is
the Heisenberg group.

The cohomology of groups has played an important role in various aspects of topology, number theory,
and algebraic geometry.  Our initial interest was generated by the foundational work of D. Quillen \cite{Q1}, \cite{Q2}
and the connections with algebraic K-theory as also developed by Quillen \cite{Q3}.  Subsequently, thanks
to the work of many mathematicians beginning with J. Alperin - L. Evens \cite{AE}, J. Carlson \cite{Ca},
and E. Cline - B. Parshall - L. Scott  \cite{CPS}, cohomology of groups has evolved into a useful
tool (``support varieties") for the study of representations of finite group schemes.  We would be amiss not to mention
work of the author with B. Parshall (e.g., \cite{FPar1}), A. Suslin (e.g., \cite{FS}), and J. Pevtsova 
(e.g., \cite{FP1}).  

Considerable progress has been made in computing the cohomology of infinitesimal group schemes of height 1, beginning with
work of Friedlander-Parshall \cite{FPar1}; this was extended by C. Drupieski - D. Nakano - N. Ngo \cite{DNN} 
to a determination of the cohomology algebra  $H^*((U_J)_1,k)$ for unipotent radicals $U_J$ of parabolic subgroups 
of reductive groups (as considered in this paper) for sufficiently large primes $p$, and further 
investigated by J. Carlson - D. Nakano \cite{CN} for small primes.  The focus of this paper is the challenge 
of achieving explicit computations of the cohomology of infinitesimal group schemes of height $> 1$.

The work
of A. Suslin, C. Bendel, and the author \cite{SFB1}, \cite{SFB2} gives a qualitative description of $H^*(G_{(r)},k)$
for the cohomology of Frobenius kernels of any linear algebraic group $G$.   A key step in this description for $G$
utilizes universal classes for $GL_N$ constructed in \cite{FS} leading to a map $\ol \phi_{GL_N}: k[V_r(GL_N)] \to
H^*((GL_N)_{(r)},k)$, complementing in the special case of $G = GL_N$, the canonical map $\psi: H^*(G_{(r)},k) \to k[V_r(G)])$; 
here, $k[V_r(G)]$ is the coordinate algebra of the variety of 1-parameter subgroups
$\bG_{a(r)} \to G$.   

In Section \ref{sec:phiU}, we extend this construction to various linear algebraic groups
$G$ with an embedding $G \to GL_N$ of exponential type.  In particular, Theorem \ref{thm:coh-exist} establishes
a map $\phi_{G,r}$ from an explicit symmetric algebra determined by $\fg = Lie(G)$ to $H^*(G_{(r)},k)$ whose image 
is highly non-trivial.  For $U_J$ the unipotent radical of a parabolic subgroup of $GL_N$, Proposition \ref{prop:coh-UN} 
shows that $\phi_{U_J,r}$ determines the map $\ol \phi_{U_J,r}$ satisfying the properties established for 
$\ol \phi_{GL_N,r}$ in \cite{SFB1}, \cite{SFB2}.

The author's original motivation for the study of $H^*(U_{(r)},k)$ was to investigate the feasibility of a cohomology-based
theory of support varieties for a unipotent algebraic group $U$ complementing his theory using 1-parameter subgroups \cite{F1}
for much more general linear algebraic groups.  One reason for restricting our attention to unipotent algebraic groups
 is that $H^*(G,k)$ is trivial for many linear algebraic groups (for example, if $G$ is simple) but is never trivial if $G$ 
 is unipotent.
Our computations show such a cohomology-based theory even for unipotent linear algebraic groups is unlikely; see, for example, 
Corollary \ref{trivial-limit} in conjunction with Theorem \ref{thm:U3r}.

We present in Section \ref{sec:stable} various results concerning the limiting behavior of the cohomology of
Frobenius kernels $U_{(r)}$ of a unipotent algebraic group $U$ as $r$ increases.  In particular, Theorem
\ref{thm:inj} established that the inverse limit with respect to $r$ of such Frobenius kernels is additively 
isomorphic to the rational cohomology of $U$; this verification is the first of several occasions in this paper
where we utilize the Andersen-Jantzen spectral sequence \cite{AJ} for the cohomology of a  connected group scheme.
The results of this Section \ref{sec:stable} formulate the general principal that the explicit cohomology classes 
considered in other sections of this paper ``vanish in the limit."

A second motivation for our calculations is internal, within the general framework of 
cohomology of groups.   Let $G$ be a simple algebraic group, $U_J \subset P_J \subset G$ the unipotent
radical of a parabolic subgroup, and $\{ \Gamma_v, v \geq 1 \}$ the descending central series of $U_J$.
This descending central series is well described by H. Azad, M. Barry, and G. Seitz in \cite{ABS}.
Using the Lyndon-Hochschild-Serre spectral sequence \cite{HS}
for the central extension $1 \to \Gamma_2/\Gamma_3 \to U_J \to U_J/\Gamma_2 \to 1$,
we construct in Definition \ref{defn:eta} the map of graded $k$-algebras
$$\eta_{U_J/\Gamma_3}: S^*((U_J/\Gamma_3)_{(r)}) \ \to \ H^*((U_J/\Gamma_3)_{(r)},k)$$
where $S^*((U_J/\Gamma_3)_{(r)})$ is a polynomial algebra with generating subspaces  \\
$\oplus_{\ell=0}^{r-1} (\fu_J/\gamma_2)^{\#(\ell+1)}[2]$ (in cohomology degree 2,  
where the Frobenius twist $(-)^{(\ell+1)}$ indicates the torus action) and 
$\oplus_{\ell=0}^{r-1} (\gamma_2/\gamma_3)^{\#(r)}[2p^{r-\ell-1}]$.
The subtlety here is the existence of the choice of an intrinsic map $\eta_{U_J/\Gamma_3}$.  
This is established with the help of the  Andersen-Jantzen spectral sequence.

Proposition \ref{prop:oleta3} verifies that $\eta_{U_J/\Gamma_3}$
factors through $\ol \eta_{U_J/\Gamma_3,r}: \ol S^*((U_J/\Gamma_3)_{(r)}) \to H^*((U_J/\Gamma_3)_{(r)},k)$,
reflecting the kernel of the inflation map $H^*((U_J/\Gamma_2)_{(r)},k) \to H^*((U_J/\Gamma_3)_{(r)},k)$; this kernel 
has explicit generators given by (\ref{reln:S2}) arising from a non-zero differential in the spectral sequence.
A key ingredient in this construction is the action of the Steenrod algebra on 
the LHS spectral sequence which enables us to identify permanent cycles. 

We view the map $\ol \eta_{U_J/\Gamma_3,r}$ as a good ``explicit" model for
$H^*((U_J/\Gamma_3)_{(r)},k)$.   As summarized in Theorem \ref{thm:U3r},  
$\ol \eta_{U_3,r} \ = \ \ol \phi_{U_3,r}:  \ol S^*((U_3)_{(r)}) \to H^*((U_3)_{(r)},k)$ 
is a map from an integrally closed domain with known generators and relations 
to $H^*((U_3)_{(r)},k)$ which is a.) injective, b.) surjective onto $p$-th powers, and 
c.) has associated graded map $gr(\ol \eta_{U,r})$ which is both injective and surjective 
onto $p$-th powers.  

We anticipate that similar arguments should apply to $H^*((U_J/\Gamma_{v+1})_{(r)},k)$ for any $v \geq 2$ 
once a suitable action of the Steenrod algebra on the Andersen-Jantzen spectral sequence \cite{AJ} is established.  
This is only one of the many challenges left unanswered in the present paper, and appears as Question 3 in the list of 
seven questions given in Section \ref{sec:questions}.

In what follows, $k$ denotes an algebraically closed field  of characteristic $p > 2$.  
We denote by $H^*(G,k)$ the (rational, or Hochschild) cohomology of an affine group scheme 
$G$ over $k$ and by $H^\bu(G,k) \ \subset \ H^*(G,k)$ the commutative subalgebra of cohomology 
classes of even degree.  We use $V^\#$ to denote the $k$-linear dual of a $k$-vector space $V$.  Other 
than in Section \ref{sec:stable} where we
consider the effect of increasing $r$, we fix an arbitrary positive integer $r$.

We thank Robert Guralnick for helpful discussions.  We especially express our gratitude and admiration
to the patient referee for detailed and constructive corrections.

%%%%%%%%%%%%%%%%%%%%%%%%%%%%%%%
%%%%%%%%%%%%%%%%%%%%%%%%%%%%%%%%%

\section{1-parameter subgroups, exponential type, and cohomology}
\label{sec:phiU}

In this section, we extend the formulation of the map of $k$-algebras
\begin{equation}
\label{def:phi-GL}
\phi_{GL_N,r}: S^*(\oplus_{\ell = 0}^{r-1} (\gl_N^{\#(r)}[2p^{r-\ell-1}]) ) \ \to \ H^\bu(GL_{N(r)},k)
\end{equation}
given by A. Suslin and the author in \cite{FS} to the Frobenius kernels $U_{(r)}$ of various unipotent subgroups
$U \subset GL_N$.  In contrast to our subsequent constructions, this extension involves little computation.
Throughout this discussion, $r$ will denote an arbitrary positive integer.  The main result of this section, 
Theorem \ref{thm:coh-exist} gives sufficient conditions for a simple group $G$ given together with an embedding
$i: G \subset GL_N$ to admit an induced  map $\phi_{G,r}$  which in turn determine an induced map for certain 
unipotent subgroups $U_J \subset G$.

Recall that for any linear algebraic group $G$ over $k$, 
the $r$-th iterate $F^r: G \to G^{(r)}$ of the Frobenius map $F: G \to G^{(1)}$ admits a 
scheme theoretic kernel $G_{(r)}  \equiv  ker\{ F^r \}$ which is an infinitesimal group scheme of height $r$.  
The coordinate algebra $k[G_{(r)}]$ of $G_{(r)}$ equals the
finite dimensional commutative Hopf algebra $k[G]/I^{p^r}$, where $I$ is the maximal ideal at the identity of $G$
and where $I^{p^r}$ is the ideal generated by $\{ f^{p^r}, f \in I \}$.  
A (rational) $G_{(r)}$-module is a comodule for $k[G_{(r)}]$ or, equivalently, a module for 
$kG_{(r)} \ \equiv \ (k[G_{(r)}])^\#$ (the $k$-linear dual of  $k[G_{(r)}]$ with its inherited Hopf algebra structure).  
For $G$ defined over $\bF_{p^r}$, we may view the Frobenius map $F^r$
as an endomorphism of $G$ and  $G_{(r)} \subset G$ as the kernel of $F^r: G \to G$.

The universal, $GL_N$-invariant classes $e_{r-\ell}  = e_{r-\ell}^{(0)} \in H^{2p^{r-\ell-1}}(GL_N,\gl_N^{(r-\ell)})$ of \cite{FS} 
and their $\ell$-th Frobenius twist $e_{r-\ell}^{(\ell)} \in H^{2p^{r-\ell-1}}(GL_N,\gl_N^{(r)})$ (i.e., pull-back 
along the $\ell$-th iterate $F^\ell$ of the Frobenius morphism $F: GL_N \to GL_N$) are elements in the rational
cohomology of the (reductive) algebraic group $GL_N$.   The restriction of $e_{r-\ell}^{(\ell)}$ to $GL_{N(r)}$, 
\begin{equation}
\label{FS:er}
(e_{r-\ell}^{(\ell)})_{(r)} \ \in \ H^{2p^{r-\ell-1}}(GL_{N(r)},\gl_N^{(r)}) \ \simeq \ H^{2p^{r-\ell-1}}(GL_{N(r)},k) \otimes \gl_N^{(r)},
\end{equation}
can be identified with a $GL_N$-equivariant map 
$$\gl_N^{\#(r)}[2p^{r-\ell-1}] \to H^{2p^{r-\ell-1}}(GL_{N(r)},k)$$
(vanishing on the dual trace class $Tr^{(r)} \in \gl_N^{\#(r)}$), thereby determining the $GL_N$-equivariant map of
commutative $k$-algebras (\ref{def:phi-GL}).  For $\ell < r$, the Frobenius map $F^\ell$ restricts to 
$F^\ell: GL_{N(r)} \to GL_{N(r)}$ and factors as 
$$GL_{N(r)} \twoheadrightarrow GL_{N(r)}/GL_{N(\ell)} \simeq GL_{N(r-\ell)} \subset GL_{N(r)}.$$
The Frobenius twist $(e_{r-\ell}^{(\ell)})_{(r)}$ can thus be realized as the pull-back along 
$GL_{N(r)} \twoheadrightarrow GL_{N(r-\ell)}$ of 
$(e_{r-\ell}^{(0)})_{(r-\ell)} \ \in \ H^{2p^{r-\ell-1}}(GL_{N(r-\ell)},\gl_N^{(r-\ell)}).$

The following proposition summarizes some of the basic properties of these universal classes (whose
proofs can be found in \cite{FS} and \cite{SFB1}).

\begin{prop}
\label{prop:basic}
With notation as above,
\begin{enumerate}
\item
$(e_{r-\ell}^{(\ell)})_{(\ell+1)} \not= 0 \ in \ H^{2p^{r-\ell-1}}(GL_{N(\ell+1)},\gl_N^{(r)})$.
\item
$(e_{r-\ell}^{(\ell)})_{(\ell)} = 0 \ in \ H^{2p^{r-\ell-1}}(GL_{N(\ell)},\gl_N^{(r)})$.
\item 
$e_{r-\ell} \in H^{2p^{r-\ell-1}}(GL_N,\gl_N^{(r-\ell)})$ restricts via the ``standard inclusion" $GL_{N-1} 
\subset GL_N$ to $e_{r-\ell} \in H^{2p^{r-\ell-1}}(GL_{N-1},\gl_{N-1}^{(r-\ell)})$.
\item
For any root subgroup $E_{i,j}: \bG_a \to GL_N$ (with $i < j$), the restriction of 
$e_{r-\ell}^{(\ell)}$ to $H^{2p^{r-\ell-1}}(\bG_{a(1)},\gl_N^{(r)})$ equals
$x_1^{p^{r-\ell-1}} \otimes X_{i,j}^{(r)}$, where $x_1 \in H^2( \bG_{a(1)},k)$
is the ``canonical generator".
\end{enumerate}
\end{prop}

Following \cite{SFB1} (specifically, the notation of the proof of Proposition 5.1 of \cite{SFB1}), 
we use the following notation:  we identify  $S^*(\oplus_{\ell = 0}^{r-1} (\gl_N^{\#(r)}[2p^{r-\ell-1}]))$
with the algebra of functions  on the affine space $\bA^{rN^2} = \prod_{\ell = 0}^{r-1} M_{n,n}$, identifying 
$X^{i,j}(\ell) \in  \gl_N^{\#(r)}[2p^{r-\ell-1}]$
with the $(i,j)$ coordinate function of the $\ell$-th factor.  

For any affine group scheme $G$
over $k$, we use the notation $V_r(G)$ for 
the affine scheme of  1-parameter subgroups of $G$ of height $r$ (i.e., homomorphisms $\bG_{a(r)} \to G$
of group schemes over $k$) with 
coordinate algebra $k[V_r(G)]$ as in \cite{SFB1}.

We state two theorems of Suslin-Friedlander-Bendel.  The first originates from the observation in 
\cite{SFB1} that every infinitesimal 1-parameter subgroup $\psi: \bG_{a(r)} \to GL_N$ is uniquely of the form
$$exp_{\ul B} \ \equiv \ \prod_{s=0}^{r-1} exp_{B_s} \circ F^s:  \bG_{a(r)} \to GL_N, 
\quad \ t \mapsto \prod_{s=0}^{r-1} exp(t^{p^s}\cdot B_s)$$
for some $r$-tuple $\ul B = (B_0, \ldots B_{r-1})$ of $p$-nilpotent, pair-wise commuting elements of $\gl_N$.

\begin{thm} (\cite[5.1]{SFB1})
\label{thm:SFB-factor}
The map $\phi_{GL_N,r}$ of (\ref{def:phi-GL}) factors as
\begin{equation}
\label{def:phi-bar-GL}
\phi_{GL_N,r} \ = \ \ol \phi_{GL_N,r} \circ q: S^*(\oplus_{\ell = 0}^{r-1} (\gl_N^{\#(r)}[2p^{r-\ell-1}])) \ 
\twoheadrightarrow \ k[V_r(GL_N)] \ \to 
H^\bu(GL_{N(r)},k).
\end{equation}
Here, $q$ is the quotient by the ideal generated by the relations 
$$R_{i,j,\ell,\ell^\prime} \quad = \quad \sum_t X^{i,t}(\ell)\cdot X^{t,j}(\ell^\prime) - X^{i,t}(\ell^\prime)\cdot X^{t,j}(\ell)$$
$$ S_{i,j,\ell} \quad = \quad \sum_{t_1,\ldots ,t_{p-1}} X^{i,t_1}(\ell) \cdot X^{t_1,t_2}(\ell) \cdots X^{t_{p-1},j}(\ell)$$
for all $i,j,\ell,\ell^\prime$ as in \cite[5.1]{SFB1}.  Thus, $V_r(GL_N)$ is identified with the $k$-scheme of $r$-tuples
of $p$-nilpotent matrices (in view of relations $\{ S_{i,j,\ell} \}$) which are pair-wise commuting (in view of relations
$\{ R_{i,j,\ell,\ell^\prime} \}$).
\end{thm}

Theorem \ref{thm:SFB-factor} can be viewed as a complement (in the special case $G = GL_{N(r)}$) to the 
following theorem.

\begin{thm} \cite[5.2]{SFB2}, \cite[5.2]{SFB1}
\label{thm:psi}
Fix some integer $r \geq 1$.  Then for any infinitesimal group scheme $H$ of height $\leq r$, there
is a natural homomorphism of commutative $k$-algebras
\begin{equation}
\label{psi-map}
\psi: H^\bu(H,k) \ \to \ k[V_r(H)]
\end{equation}
whose kernel is nilpotent and whose image contains all $p^r$-th powers of elements of $k[V_r(H)]$.
If $H = G_{(r)}$, the $r$-th Frobenius kernel of some linear algebraic group $G$, then we denote
$\psi$ by $\psi_{G,r}: H^\bu(G_{(r)},k) \to k[V_r(G)]$.

The map $\psi_{G,r}$ is $G$-equivariant.

In the special case of $G = GL_{N(r)}$, the composition 
$$\psi_{GL_{N,r}} \circ \ol \phi_{GL_N,r}: k[V_r(GL_N)] \ \to \ k[V_r(GL_N)]$$
is the $r$-th iterate of the Frobenius map.  In particular, $\psi_{GL_{N,r}} (\phi_{GL_N,r}(X^{i,j}(\ell))) = (X^{i,j}(\ell))^{p^r}$.
\end{thm}

\begin{remark}
\label{rem:psi}
The assertion of Theorem \ref{thm:psi} of $G$-equivariance of $\psi_{G,r}$ arises from the naturality of $\psi$,
in particular the commutativity of the first displayed square of the proof of Theorem 1.14 of \cite{SFB1}.

As shown in \cite{SFB1}, $V_r(G)$ has a natural grading given by the monoid action of (right) composition by
$V_r(\bG_{a(r)})$ on $V_r(G)$; namely, one restricts this action of $\bA^r \simeq V_r(\bG_{a(r)}$) to the 
linear polynomials $\bA^1 \subset \bA^r$.  With this grading, $X^{i,j}(\ell) \in k[V_r(GL_N)]$ has grading 
$p^{r-\ell-1}$ mapping via $\ol \phi_{GL_N,r}$ to a cohomology class of degree $2p^{r-\ell-1}$, then further mapping
via $\psi_{GL_{N,r}}$ to $F^r(X^{i,j}(\ell))$ with degree $p^r\cdot p^{r-\ell-1}$.
\end{remark}

We recall that a closed embedding $G \to GL_N$ of a linear algebraic group $G$ is said to be of {\bf exponential type} if 
the map of schemes $exp: \bG_a \times \cN_p(\gl_N) \to V_r(GL_N)$ given by the usual truncated exponential map 
restricts to $\cE: \bG_a \times \cN_p(\fg) \to V_r(G)$.  (As usual, for any $p$-restricted Lie algebra $\fg$, we denote the
$p$-operator as $(-)^{[p]}: \fg \to \fg$ and we denote by $\cN_p(\fg) \subset \fg$ the subvariety whose $k$ points are  
elements $X \in\fg$ such that $X^{[p]} = 0$.)    For any
$G$ equipped with such an embedding, every infinitesimal 1-parameter subgroup
$\bG_{a(r)} \to G$ is uniquely of the form $\prod_{s=0}^{r-1} \cE_{B_s} \circ F^s$ for some $\ul B \in \cC_r(\cN_p(\fg))$
by  \cite[2.5]{S3};
here, $\cC_r(\cN_p(\fg))$ is
the variety of $r$-tuples $(B_0,\ldots,B_{r-1})$ of $p$-nilpotent, pairwise commuting elements of $\fg$. 
.

\begin{prop}
\label{prop:def-relns}
Let $i: G \to GL_N$ be a closed embedding of exponential type for some linear algebraic group $G$.
Then the following square is a cartesian square of closed immersions
\begin{equation}
\label{cartesian}
\begin{xy}*!C\xybox{%
\xymatrix{
V_r(G) \ar[r]^{i_V} \ar[d] & V_r(GL_N) \ar[d] \\
\fg^{\times r} \ar[r]^{i} & \gl_N^{\times r}.
 }
}\end{xy}
\end{equation}

In other words, we have a cocartesian square of quotient maps of $k$-algebras
\begin{equation}
\label{cocartesian}
\begin{xy}*!C\xybox{%
\xymatrix{
S^*(\oplus_{\ell = 0}^{r-1} (\gl_N^{\#(r)}[2p^{r-\ell-1}])) \ar[d]^{i^*}  \ar[r]^-{q_{GL_N,r}}
& k[V_r(GL_N)] \ar[d]^{i_V^*}  \\
S^*(\oplus_{\ell = 0}^{r-1} (\fg^{\#(r)}[2p^{r-\ell-1}])) \ar[r]^-{q_{G,r}}
& k[V_r(G)].
 }
 }\end{xy}
\end{equation}
\end{prop}

\begin{proof}
The condition that $i: G \to GL_N$ is of exponential type enables us to identify
the embedding $V_r(G) \to V_r(GL_N)$ of schemes representing height $r$ infinitesimal
1-parameter subgroups with the embedding of schemes of $r$-tuples of $p$-nilpotent, pair-wise 
commuting elements of respective Lie algebras.  Using this, we verify that (\ref{cartesian})
arises as a cartesian square of representable functors.  Namely, we verify that the
defining relations $\{ R_{i,j,\ell,\ell^\prime}, S_{i,j,\ell} \}$ in 
$S^*(\oplus_{\ell = 0}^{r-1} (\gl_N^{\#(r)}[2p^{r-\ell-1}]))$ (for \\
$V_r(GL_N) \subset (\gl_N)^{\times r}$) 
have image in $S^*(\oplus_{\ell = 0}^{r-1} (\fg^{\#(r)}[2p^{r-\ell-1}]))$ (i.e., restrictions to 
$(\fg)^{\times r} \subset (\gl_N)^{\times r}$) which generate defining relations for $V_r(G) \subset (\gl_N)^{\times r}$.
This follows from the observation for $X, Y \in \fg$ that the condition that $X,Y$ commute in $\fg$ is the 
same as the condition that their images commute in $\gl_N$, and the condition that $X^{[p]} = 0$ is the condition
that the image of $X$ in $\gl_N$ has $p$-th power 0.
 
The cartesian square  (\ref{cartesian}) is equivalent to the cocartesian square (\ref{cocartesian})
of coordinate algebras thanks to the anti-equivalence of categories relating affine $k$-schemes and
finitely generated commutative $k$-algebras.
\end{proof}

We thank R Guralnick for explaining the following result of S. Garibaldi given in \cite{Garib}.  

\begin{prop} \cite[Prop 8.1]{Garib} 
\label{prop:Garib}
If $G$ is a simple algebraic group for which $p>2$ is a very good prime (i.e., 
for type $A_{n-1}$, $p$ does not divide $n$; for type $G_2, F_4, E_6, E_7$, $p > 3$; for type $E_8$,  $p > 5$),
then there exist a closed embedding $i: G \to GL_N$ such that the induced map $i: \fg \to \gl_N$ admits
a unique $G$-equivariant splitting $\tau: \gl_N \to \fg$.
\end{prop}

Garibaldi's result is proved for a split, almost simple algebraic group over an arbitrary field $F$ and uses
an embeddding $G \subset GL_N$ associated to a representation defined over $F$.  The splitting
is also defined over $F$.

The following theorem can be interpreted as giving a lower bound on the ``size" of $H^\bu(G_{(r)},k)$.

\begin{thm}
\label{thm:coh-exist}
Let $G$ be a simple algebraic group and assume that $p > 2$ is very good for $G$.   Assume give some 
embedding $i: G \subset GL_N$ as in Proposition \ref{prop:Garib} which is also an embedding of exponential type.  Let
$\tau: \gl_N \to \fg = Lie(G)$ be the unique $G$-equivariant splitting of $i: \fg \to \gl_N$.     
Set \ $\phi_{G,r} \ = \ i^* \circ \phi_{GL_N,r} \circ \tau^*$.  Then
$$\psi_{G,r} \circ \phi_{G,r} : S^*(\oplus_{\ell = 0}^{r-1} (\fg^{\#(r)}[2p^{r-\ell-1}])) \to  H^\bu(G_{(r)},k) \to k[V_r(G)]$$
has image containing all $p^r$-th powers. 
%and is injective when restricted to each $\fg^{\#(r)}[2p^{r-\ell-1}]$.

Consider the unipotent radical $U_J$ of $P_J \subset G$  
for some subset $J$ of the set of fundamental positive roots $\Pi$ of a chosen root system for $G$ (given by a
choice $B \subset G$ of Borel subgroup with maximal torus $T$), denote by
$\tau_{U_J}: \fg \to \fu_J$ the unique $T$-equivariant splitting of $i_{U_J}: \fu \to \fg$, and set $\phi_{U_J,r} = 
 i_{U_J}^* \circ \phi_{G,r} \circ \tau_{U_J}^*$.  Then the composition 
$$\psi_{U_J,r} \circ \phi_{U_J,r}: S^*(\oplus_{\ell = 0}^{r-1} (\fu_J^{\#(r)}[2p^{r-\ell-1}])) \to H^\bu((U_J)_{(r)},k) \to k[V_r(U_J)]$$
has image containing all $p^r$-th powers and is injective when restricted to each $\fu_J^{\#(r)}[2p^{r-\ell-1}]$.
\end{thm}

\begin{proof}
Consider the following diagram, where $F^r$ denotes the $r$-th power of the Frobenius map:
\begin{equation}
\label{commute:exptype}
\begin{xy}*!C\xybox{%
\xymatrix{
S^*(\oplus_{\ell = 0}^{r-1} (\gl_N^{\#(r)}[2p^{r-\ell-1}]))  \ar[r]^-{q_{GL_N,r}}
 \ar@/^2pc/ [rr]^-{\phi_{GL_N,r}} 
& k[V_r(GL_N)] \ar[d]_{i_V^*} \ar[r]^-{\ol \phi_{GL_N,r}}   \ar@/^2pc/ [rr]^-{F^r}
& H^\bu(GL_{N(r)},k) \ar[d]^{i^*} \ar[r]^{\psi_{GL_N,r}} & k[V_r(GL_N)] \ar[d]_{i_V^*}\\
S^*(\oplus_{\ell = 0}^{r-1} (\fg^{\#(r)}[2p^{r-\ell-1}])) \ar[r]^-{q_{G,r}} \ar[u]^{\tau^*}
 \ar@/_2pc/ [rr]_-{\phi_{G,r}} & k[V_r(G)] \ar@/_2pc/ [rr]_-{F^r} 
& H^\bu(G_{(r)},k) \ar[r]^{\psi_{G,r}} & k[V_r(G)]  
 }
}\end{xy}
\end{equation}
Commutativity of the left square is a consequence of Proposition \ref{prop:def-relns} and commutativity of
the right square follows from the functoriality of $\psi$; commutativity of the left rectangle follows from
the definition of   $\phi_{G,r}$, whereas commutativity of the right rectangle is a consequence of the 
functoriality of $F^r$ for maps of varieties defined over $\bF_{p^r}$ (granted that Garibaldi's splitting is
defined over $\bF_p$, as observed after the statement of Proposition \ref{prop:Garib}).

To show that $\psi_{G,r} \circ \phi_{G,r}$ has image containing all $p^r$-th powers, we observe that 
$F^r: k[V_r(G)] \to k[V_r(G)]$ has image containing all $p^r$-th powers.  Thus, a simple diagram chase around the commutative
diagram (\ref{commute:exptype}) using the surjectivity of $q_{G,r}$ verifies this assertion.

Consider now the 
 following diagram obtained by replacing $i: G \to GL_N$ in (\ref{commute:exptype}) by $i_{U_J}: U_J \to G$:
 \begin{equation}
\label{commute:Utype}
\begin{xy}*!C\xybox{%
\xymatrix{
S^*(\oplus_{\ell = 0}^{r-1} (\fg^{\#(r)}[2p^{r-\ell-1}])) \ar[r]^-{q_{G,r}}
 \ar@/^2pc/ [rr]^-{\phi_{G,r}} & k[V_r(G)] \ar[d]^{i_V^*} \ar@/^2pc/ [rr]^-{F^r} 
& H^\bu(G_{(r)},k) \ar[d]^{i^*} \ar[r]^{\psi_{G,r}} & k[V_r(G)] \ar[d]^{i_V^*}  \\
S^*(\oplus_{\ell = 0}^{r-1} (\fu_J^{\#(r)}[2p^{r-\ell-1}])) \ar[r]^-{q_{U_J,r}} \ar[u]^{\tau_{U_J}^*}
 \ar@/_2pc/ [rr]_-{\phi_{U_J,r}} & k[V_r(U_J)] \ar@/_2pc/ [rr]_-{F^r} 
& H^\bu((U_J)_{(r)},k) \ar[r]^{\psi_{U_J,r}} & k[V_r(U_J)].
 }
}\end{xy}
\end{equation}
The previous argument applies with only notational changes to verify the corresponding assertions for 
$\psi_{U_J,r} \circ \phi_{U_J,r}$.
\end{proof}

\begin{ex}
\label{ex:exp-type}
As stated in  \cite[1.8]{SFB1}, the classical simple algebraic groups $Sp_{2n}, \ SO_n$ and 
$SL_n$  admit  embeddings of exponential type.  Namely,  
one considers a vector space (of dimension $2n$ for $Sp_{2n}$, of dimension $n$ for $O_n$) equipped 
with a non-degenerate bilinear form and one takes the embedding given by considering those linear 
isomorphisms preserving the form.   These embeddings $i: G \to GL_N$ are defined over $\bF_p$ 
and also satisfy the condition that  
$i: \gl_N \to \fg$ admits a (unique) $G$-equivariant splitting.
\end{ex}

For $U_J \subset GL_N$, we have the following natural strengthening of Theorem \ref{thm:coh-exist}.
We denote by $T_N \subset GL_N$ the maximal torus of diagonal matrices.

\begin{prop}
\label{prop:coh-UN}
Let $i_{U_J}: U_J \to GL_N$ be the inclusion of the unipotent radical of a parabolic subgroup
of $GL_N$.  Then we have the following $T_N$-equivariant, commutative diagram, a stronger version of the 
diagram obtained from (\ref{commute:exptype}):
\begin{equation}
\label{commute:exptype-UN}
\begin{xy}*!C\xybox{%
\xymatrix{
S^*(\oplus_{\ell = 0}^{r-1} (\gl_N^{\#(r)}[2p^{r-\ell-1}]))  \ar[r]^-{q_{GL_N,r}} \ar[d]^{i_{U_N}^*}
 \ar@/^2pc/ [rr]^-{\phi_{GL_N,r}} 
& k[V_r(GL_N)] \ar[d]_{i_V^*} \ar[r]_-{\ol \phi_{GL_N,r}}   \ar@/^2pc/ [rr]^-{F^r}
& H^\bu(GL_{N(r)},k) \ar[d]^{i_{U_N}^*} \ar[r]^{\psi_{GL_N,r}} & k[V_r(GL_N)] \ar[d]_{i_V^*}\\
S^*(\oplus_{\ell = 0}^{r-1} (\fu_J^{\#(r)}[2p^{r-\ell-1}])) \ar[r]^-{q_{U_J,r}} 
 \ar@/_2pc/ [rr]_-{\phi_{U_J,r}} & k[V_r(U_J)] \ar@/_2pc/ [rr]_-{F^r} \ar[r]^{\ol \phi_{U_J,r}}
& H^\bu((U_J)_{(r)},k) \ar[r]^{\psi_{U_J,r}} & k[V_r(U_J)]  
 }
}\end{xy}
\end{equation}
In particular, any element in the kernel of $\ol \phi_{U_J,r}$ has $p^r$-th power 0.
\end{prop}

\begin{proof}
Observe that the $T$-weights of $H^\bu((U_J)_{(r)},k)$ are all negatives of roots in the lattice generated by
the roots of $U_N$, so that $i_{U_J}^* \circ \phi_{GL_N,r}$ restricted to each $\fu_J^{\#(r)}[2p^{r-\ell-1}]$ must
factor as a $T$-equivariant map through the projection $\gl_N^{\#(r)}[2p^{r-\ell-1}] \to \fu_J^{\#(r)}[2p^{r-\ell-1}]$.
This determines the map $\phi_{U_J,r}$ making commutative the left rectangle (i.e., double square) of 
(\ref{commute:exptype-UN}).   In view of the surjectivity of $q_{GL_N,r}$,
the map $\ol \phi_{U,r}$ is uniquely determined making the middle square commute.

Recall that the composition of $F$ with the geometric Frobenius is the $p$-th power map.  This implies
that any element in the kernel of $F^r$, and thus any element in the kernel of $\ol \phi_{U,r}$, has $p^r$-th 
power 0.
\end{proof}

The next proposition helps to pin down $\phi_{G,r}$.

\begin{prop}
\label{prop:phiGamma2}
Let $i_{U_J}: U_J \to GL_N$ be the inclusion of the unipotent radical of a parabolic subgroup
of $GL_N$.  Denote by 
$\Gamma_2 \subset U_J$ the commutator subgroup of $U_J$
with Lie algebra $\gamma_2$, and denote by $q: U_J \to U_J/\Gamma_2$ the projection.
Then the following square commutes for each $\ell, \ 0 \leq \ell < r$:
\begin{equation}
\label{square:Gamma2}
\begin{xy}*!C\xybox{%
\xymatrix{
\fu_J^{\#(r)}[2p^{r-\ell-1}] \ar[r]^= & \fu_J^{\#(r)}[2p^{r-\ell-1}]  \ar[r]^-{\phi_{U_J,r}}  & H^{2p^{r-\ell-1}}((U_J)_{(r)},k) \\
(\fu_J/\gamma_2)^{\#(r)}[2p^{r-\ell-1}] \ar[u]^{q^*}\ar[r]   \ar@/_2pc/ [rr]_-{\phi_{U_J/\Gamma_2,r}}
 & S^{p^{r-\ell-1}}((\fu_J/\gamma_2)^{\#(\ell+1)}[2]) \ar[r]& 
H^{2p^{r-\ell-1}}((U_J/\Gamma_2)_{(r)},k)  \ar[u]^{q^*}.
}
 }\end{xy}
\end{equation}
The lower left map of (\ref{square:Gamma2}) is a special case (for $V = (\fu_J/\gamma_2)^{\#(\ell+1)}$)
of the natural embedding $V^{(r-\ell-1)} \subset S^{p^{r-\ell-1}}(V)$ sending $v \in V$ to its $p^{r-\ell-1}$-st 
power.   We have abused notation by using $\phi_{U_J,r}$
to denote the restriction to the indicated summand of the map of (\ref{commute:exptype-UN}) with this name.

For a minimal weight $\alpha$ of $\fu_J$, the map 
$\phi_{U/\Gamma_2,r}$ restricted to the weight space $k\cdot X^\alpha(\ell)$ of weight $p^r\alpha$
is explicitly described as the projection $(\fu/\gamma_2)^{\#(r)}[2p^{r-\ell-1}] \to k\cdot X^\alpha{(\ell)}$,
followed by the map sending $X^\alpha(\ell)$ to the $p^{r-\ell-1}$-st power of the natural class in $H^2(\bG_{a(r)},k)$,
followed by the inflation map $H^*(\bG_{a(r)},k) \to H^*((U_J/\Gamma_2)_{(r)},k)$ induced by the quotient map
$U_J/\Gamma_2 \to \bG_a$ onto the root subgroup indexed by the root $\alpha$.
\end{prop}

\begin{proof}
We first observe the commutativity of
\begin{equation}
\label{square:Ealpha}
\begin{xy}*!C\xybox{%
\xymatrix{
\gl_N^{\#(r)}[2p^{r-1}] \ar[r]^-{\phi_{GL_N,r}} \ar[d]^{i_{U_J}^*}& H^{2p^{r-1}}(GL_{N(r)},k) \ar[d]^{i_{U_J}^*} \\
\fu_J^{\#(r)}[2p^{r-1}] \ar[r]^-{\phi_{U_J,r}} & H^{2p^{r-1}}((U_J)_{(r)},k) \\
(\fe_\alpha)^{\#(r)}[2p^{r-1}] \ar[u]^{q_\alpha^*} \ar[r]_-{\phi_{E_\alpha,r}} & H^{2p^{r-1}}((E_\alpha)_{(r)},k)  \ar[u]^{q_\alpha^*}
}
 }\end{xy}
\end{equation}
where $\bG_a \simeq E_\alpha \subset U_J$ is the root subgroup associated to a minimal root $\alpha$ and $q_\alpha: U_J/\Gamma_2 \to 
E_\alpha$ is the projection, a group homomorphism since $\alpha$ is minimal.  This commutativity of (\ref{square:Ealpha})
follows from Proposition \ref{prop:basic}(4) and the fact $i_\alpha: \bG_a \to U_J/\Gamma_2$ is left inverse to $q_\alpha$.
This implies the commutativity of (\ref{square:Gamma2}) for $\ell = 0$

For $\ell > 0$, we apply (\ref{square:Ealpha}) with $r$ replaced by $r-\ell$ and use the commutativity of
\begin{equation}
\begin{xy}*!C\xybox{%
\xymatrix{
\fu_J^{\#(r-\ell)}[2p^{r-\ell-1}] \ar[r] \ar[d]^{(-)^\ell} & H^{2p^{r-\ell- 1}}((U_J)_{(r-\ell)},k) \ar[d]^{\ell *}\\
\fu_J^{\#(r)}[2p^{r-1}] \ar[r]  & H^{2p^{r-\ell- 1}}((U_J)_{(r)},k) 
}
 }\end{xy}
\end{equation}
which is a direct consequence of the definition of $\phi_{GL_N,r}$.
\end{proof}

\vskip .2in

%%%%%%%%%%%%%%%%%%%%%%%%%%
%%%%%%%%%%%%%%%%%%%%%%%%%%%
%%%%%%%%%%%%%%%%%%%%%%%%%%%%

\section{Stabilization of $H^\bu(U_{(r)},k)$ with respect to $r$}
\label{sec:stable}

Part of the author's motivation for considering $H^*(U_{(r)},k)$ was
 the hope that some form of ``continuous cohomology" for the unipotent
algebraic group $U$ would prove useful in the study of the (rational) representations
of $U$.  This requires understanding the limiting behavior of
$H^*(U_{(r)},k)$ as $r$ increases.  Earlier computational information for  $H^*(U_{(r)},k)$   
(especially in \cite{SFB1}, \cite{SFB2}) shed little if any light on this limiting behavior.

We begin by recalling the following spectral sequence formulated by H. Andersen and J. Jantzen.

\begin{prop} \cite{AJ}, \cite[\S 9]{J}
\label{prop:AJ}
Let $H$ be an irreducible affine group scheme (over $k$) and let $I_1 \subset k[H]$
denote the maximal ideal at the identity of $H$.  The filtration of $k[H]$ by powers of $I_1$
leads to an associated graded Hopf algebra which is the coordinate 
algebra of the vector group scheme $gr(H)$.  For any rational $H$-module $M$,
there is a naturally associated convergent spectral sequence
\begin{equation}
\label{specAJ}
{}^{AJ}E^{i,j}_1(H) \ = \ H^{i+j}(gr(H),k)_i \otimes M \ \Rightarrow \ H^{i+j}(H,M),
\end{equation}
where $H^*(gr(H),k)_i$ is the cohomology algebra of the $i^{th}$ graded summand of 
the Hochschild complex of $gr(H)$.

If  $G$ is a  irreducible linear algebraic group (hence, a reduced affine group scheme of finite
type over $k$) and $p \not= 2$, then ${}^{AJ} E_1^{i,j}(G)$ can be identified with the direct sum of
tensor products of the form
\begin{equation}
\label{summands}
S^{a_1}(\fg^{\#(1)}[2])  \otimes S^{a_2}(\fg^{\#(2)}[2]) \otimes \cdots \otimes \Lambda^{b_1}(\fg^\#[1])
\otimes \Lambda^{b_2}(\fg^{\#(1)}[1]) \otimes \cdots
\end{equation}
where the sum is over all sequences $\{ a_n \}, \{ b_n \}$ with each $a_n \geq 0$, each $b_n \geq 0$
and 
$$ i \ = \ \sum_{n\geq 1} (a_n p^n + b_n p^{n-1}), \quad \quad i+j \ = \ \sum_{n\geq 1}(2a_n + b_n). $$
Moreover, for any $r \geq 1$, ${}^{AJ}E_1^{i,j}(G_{(r)})$ can be identified with the direct sum of those
tensor products of the form (\ref{summands}) with $a_n = b_n = 0, \ n > r$.
\end{prop}

The following theorem shows for a large class of unipotent algebraic groups that the 
rational cohomology equals the ``continuous cohomology", thereby motivating the study of this
continuous cohomology.

\begin{thm}
\label{thm:inj}
Let $G$ be a simple algebraic group provided with a choice of Borel subgroup $B \subset G$
with maximal torus $T$ and unipotent radical $U$.
Let $U_1 \subset U$ be $T$-stable closed subgroup, $U_2 \subset U_1$ a $T$-stable, normal 
closed subgroup of $U_1$, and consider  $V \ \equiv \ U_1/U_2$.  The natural map
$$H^*(V,k) \ \to \ \varprojlim_r H^*(V_{(r)},k)$$
is an isomorphism.
\end{thm}

\begin{proof}
Let $0 \not= \zeta \in H^d(V,k)$ be a $T$-eigenvector of weight 
 $\omega = \sum_{m=1}^\ell w_m\alpha_m, w_m \geq 0$, an 
element in the positive cone of the root lattice for $\fv = Lie(V)$; here, $\alpha_1,\ldots,\alpha_\ell$ are the 
simple roots determined by $B \subset G$.  One verifies by inspection that
the restriction map ${}^{AJ}E_1^{*,*}(V) \ \to \ {}^{AJ}E_1^{*,*}(V_{(r)})$ is an isomorphism of $\omega$-weight spaces
$({}^{AJ}E_1^{*,*}(V))_\omega \ \to \ ({}^{AJ}E_1^{*,*}(V_{(r)})_\omega$ provided that $p^{r-1}$ is greater than any 
of $w_1,\ldots,w_\ell$.  Namely, using (\ref{summands}), we can see that this condition on each $w_i$ implies that there can be
no contribution from a tensor factor in  
$S^{a_i}(\fv^{\#(i)}[2])$ or in $\Lambda^{b_j}(\fv^{\#(j-1}[1])$ for $i > r$ or for $j  > r+1$.
Observe that the spectral sequences $\{ {}^{AJ}E_s^{i,j}(V), s \geq 1\}$ and $\{ {}^{AJ}E_s^{i,j}(V_{(r)}), \ s \geq 1\}$ 
split (additively) as a direct sum of spectral sequences indexed by the weights in the
positive cone of the root lattice for $\fv$.  

This implies that the restriction map induces an isomorphism $(H^*(V,k))_\omega \to (H^*(V_{(r)},k))_\omega$ 
whenever $p^{r-1}$ is greater than $max\{ w_i \}$, the maximum of $w_1,\ldots,w_\ell$ in the expression for $\omega$.  
Thus, the restriction map $(H^*(V,k))_\omega \to \varprojlim_r (H^*(V_{(r)},k))_\omega$ is an isomorphism for 
all weights $\omega$  in the positive cone of the root lattice for $\fv$,
so that $H^*(V,k) \ \to \ \varprojlim_r H^*(V_{(r)},k)$ is also an isomorphism.
\end{proof}

We easily verify that the injectivity statement of Theorem \ref{thm:inj} extends to cohomology of $V$ with coefficients
in a finite dimensional $V$-module having a compatible torus action.

\begin{cor}
Retain the notation of Theorem \ref{thm:inj} and let $M$ be a finite dimensional rational $V \rtimes T$-module.
Then the natural map
$$H^*(V,M) \ \to \ \varprojlim_r H^*(V_{(r)},M)$$
is injective.  

Moreover, if  \ $\zeta_r \in H^d(V_{(r)},M)$ is the restriction of some $\zeta_s \in H^*(V_{(s)},M)$
for all $s \geq r$, then there exists some $\zeta \in H^d(V,M)$ which restricts to $\zeta_r$.

\end{cor}

\begin{proof}
We repeat the argument of the proof  of Theorem \ref{thm:inj} for A-J spectral sequences
$\{ {}^{AJ}E_s^{i,j}(V,M), s \geq 1\}$ and $\{ {}^{AJ}E_s^{i,j}(V_{(r)},M), s \geq 1\}$.  The action of 
$V$ on the associated
graded group of $M$ is trivial, so that the $E_1$ terms are obtained from those for coefficients 
equal to $k$ by tensoring with $M$.  These spectral sequences now split as a direct sum of
spectral sequences indexed by weights given as the sum of a weight of $M$ and 
an element in the positive cone of the root lattice for $\fv$.  As in the proof of Theorem \ref{thm:inj},
we conclude that the restriction map ${}^{AJ}E_1^{*,*}(V,M) \ \to \ {}^{AJ}E_1^{*,*}(V_{(r)},M)$ is an isomorphism of 
$\omega$-weight spaces
$({}^{AJ}E_1^{*,*}(V,M))_\omega \ \to \ ({}^{AJ}E_1^{*,*}(V_{(r)},M))_\omega$ provided that $p^{r-1}$ does not divide 
the coefficient of the linear expansion of any weight of the form $w+w^\prime$, where $w^\prime$ 
is a weight of $M$.    The remainder of the proof is a repetition of that of Theorem \ref{thm:inj}.
\end{proof}

In order to  investigate how the map $\phi_{GL_N,r}$ in (\ref{def:phi-GL}) behaves as $r$ increases, we 
introduce in the next proposition the map $\rho$.

\begin{prop}
\label{prop:bracket}
Define the degree preserving map of graded $k$-algebras 
\begin{equation}
\label{rho-phi-gl}
\rho: S^*(\oplus_{\ell = 0}^{r-1} (\gl_N^{(r)\#}[2p^{r-\ell-1}]))  \to \ S^*(\oplus_{\ell = 0}^{r-2} (\gl_N^{(r-1)\#}[2p^{r-\ell-2}]))
\end{equation}
by sending  $X^{s,t}(\ell) \in \gl_N^{(r)\#}[2p^{r-\ell-1}]$ to 
the $p$-th power $(X^{s,t}(\ell))^p \in S^p(\gl_N^{(r-1)\#}[2p^{r-\ell-2}])$
if $\ell < r-1$ and to 0 if $\ell = r-1$.
Then $\rho$ fits in the $GL_N$-equivariant
commutative square
\begin{equation}
\label{comm:bracket}
\begin{xy}*!C\xybox{%
\xymatrix{
S^*(\oplus_{\ell = 0}^{r-1} (\gl_N^{\#(r)}[2p^{r-\ell-1}])) \ar[d]_{\rho} \ar[r]^-{\phi_{GL_N,r}} & H^\bu(GL_{N(r)},k) \ar[d]^{res}\\
S^*(\oplus_{\ell = 0}^{r-2} (\gl_N^{\#(r-1)}[2p^{r-\ell-2}])) \ar[r]_-{\phi_{GL_N,r-1}} & H^\bu(GL_{N(r-1)},k). }
}\end{xy}
\end{equation}
\end{prop}

\begin{proof}
We first show that $\phi_{GL_N,r}(X^{s,t}(0)) \in H^{2p^{r-1}}(GL_{N(r)},k)$ restricts to the $p$-th power
of $\phi_{GL_N,r-1}(X^{s,t}(0)) \in H^{2p^{r-2}}(GL_{N(r-1)},k)$.  By  \cite[3.4]{SFB1}, both of these
classes restrict to the $p^{r-1}$-st power of the image of $\phi_{GL_N,1}(X^{s,t}(0))$
in $H^2(GL_{N-1(1)},k)$.   Thus, the outer rectangle and the lower square of the following diagram
commutes:
\begin{equation}
\label{gl-diag}
\begin{xy}*!C\xybox{%
\xymatrix{
\gl_N^{\#(r)}[2p^{r-1}] \ar[d]_{\rho} \ar[r]^{e_r} & H^{2p^{r-1}}(GL_{N(r)},k) \ar[d]^{res} \\
S^p(\gl_N^{\#(r-1)}[2p^{r-2}]) \ar[d]_{(-)^{\rho^{r-2}}} \ar[r]^{e_{r-1}} & H^{2p^{r-1}}(GL_{N(r-1)},k) \ar[d]^{res} \\
S^{p^{r-1}}(\gl_N^{\#(1)}[2]) \ar[r]_{e_1} &  H^{2p^{r-1}}(GL_{N(1)},k) }
}\end{xy}
\end{equation}

The images in $H^{2p^{r-1}}(GL_{N(r-1)},k)$ of the two compositions in the upper square of (\ref{gl-diag}) are 
each irreducible $GL_N$-modules (copies of $(\gl_N^{\#(r)}[2p^{r-1}])/(k\cdot Tr^{(r)})$) which restrict non-trivially 
to $H^{2p^{r-1}}(GL_{N(1)},k)$.   Using the form of the $E_1$-term of the A-J spectral sequence (\ref{specAJ}),
we conclude that there is a unique copy of $(\gl_N^{\#(r)}[2p^{r-1}])/(k\cdot Tr^{(r)})$
in ${}^{AJ}E_1^{*,*}(GL_{N(r-1)})$ of cohomological degree $2p^{r-1}$, so that these images
are equal.  The functoriality of  (\ref{specAJ})
with respect to $GL_{N(1)} \to GL_{N(r)}$ now implies that 
the upper square of (\ref{gl-diag}) must also commute.

By definition of $e_{r-\ell}^{(\ell)}$ as the pull-back via $F^\ell: GL_N \to GL_N$ of $e_{r-\ell}$, we have the commutativity of
the following square
\begin{equation}
\label{pull-back-er}
\begin{xy}*!C\xybox{%
\xymatrix{
\gl_N^{\#(r-\ell)}[2p^{r-\ell-1}] \ar[d]_{(-)^{(\ell)}} \ar[r]^-{e_{r-\ell}} & H^{2p^{r-\ell-1}}(GL_{N(r-\ell)},k) \ar[d]^{(F^\ell)^*} \\
\gl_N^{\#(r)}[2p^{r-\ell-1}] \ar[r]^-{e_{r-\ell}^{(\ell)}} & H^{2p^{r-\ell-1}}(GL_{N(r)},k).}
}\end{xy}
\end{equation}
Consequently, pulling back via $F^{\ell}$ the commutative upper square of (\ref{gl-diag}) 
with $r$ replaced by $r-\ell$ 
determines the following commutative square  for each $\ell, 0 \leq \ell < r$:
\begin{equation}
\label{gl-diag-ell}
\begin{xy}*!C\xybox{%
\xymatrix{
\gl_N^{\#(r)}[2p^{r-\ell-1}] \ar[d]_{\rho} \ar[r]^{e_{r-\ell}^{(\ell)}} & H^{2p^{r-\ell-1}}(GL_{N(r)},k) \ar[d]^{res} \\
S^p(\gl_N^{\#(r-1)}[2p^{r-\ell-2}]) \ar[r]^{e_{r-\ell-1}^{(\ell)}} & H^{2p^{r-\ell -1}}(GL_{N(r-1)},k). }
}\end{xy}
\end{equation}

The proposition now follows since the maps of (\ref{comm:bracket}) are maps of $k$-algebras and
the commutativity of (\ref{gl-diag-ell}) implies the commutativity of (\ref{comm:bracket}) on generators.
\end{proof}

The proof of the following lemma (including the definitions of $\rho_G$ and $\rho_{U_J}$) is immediate from the
definitions.

\begin{lemma}
\label{lemma:other-rho}
Let $U_J \subset G$ satisfy the hypotheses of Theorem \ref{thm:coh-exist}.   Define 
\begin{equation}
\label{rho-phi-g}
\rho_G: S^*(\oplus_{\ell = 0}^{r-1} (\fg^{\#(r)}[2p^{r-\ell-1}]))  \to \ S^*(\oplus_{\ell = 0}^{r-2} (\fg^{\#(r-1)}[2p^{r-\ell-2}]))
\end{equation}
by sending  $X \in \fg^{\#(r)}[2p^{r-\ell-1}]$ to $X^p \in S^p(\fg^{\#(r-1)}[2p^{r-\ell-2}])$
if $\ell < r-1$ and to 0 if $\ell = r-1$.  Define
$\rho_{U_J}: S^*(\oplus_{\ell = 0}^{r-1} (\fu_J^{\#(r)}[2p^{r-\ell-1}]))  \to \ S^*(\oplus_{\ell = 0}^{r-2} (\fu_J^{\#(r-1)}[2p^{r-\ell-2}]))$
similarly.  Then $\rho, \rho_G, \rho_{U_J}$ fit in the following commutative diagram
\begin{equation}
\label{rho-compat}
\begin{xy}*!C\xybox{%
\xymatrix{
S^*(\oplus_{\ell = 0}^r(\fu_J^{\#(r)}[2p^{r-\ell-1}])) \ar[r]^{\tau_{U_J}^*} \ar[d]^{\rho_{U_J}} & 
S^*(\oplus_{\ell = 0}^r(\fg^{\#(r)}[2p^{r-\ell-1}]))
\ar[r]^{\tau_G^*} \ar[d]^{\rho_G} 
& S^*(\oplus_{\ell = 0}^r(\gl_N^{\#(r)}[2p^{r-\ell-1}])) \ar[d]_{\rho} \\
S^*(\oplus_{\ell = 0}^{r-1}(\fu_J^{\#(r-1)}[2p^{r-\ell-2}])) \ar[r]^{\tau_{U_J}^*}  & S^*(\oplus_{\ell = 0}^{r-1}(\fg^{\#(r-1)}[2p^{r-\ell-2}])) 
\ar[r]^{\tau_G^*}  & S^*(\oplus_{\ell = 0}^{r-1}(\gl_N^{\#{(r-1)}}[2p^{r-\ell-2}])). }
}\end{xy}
\end{equation}
\end{lemma}

The following extension of Proposition \ref{prop:bracket} to $G$ and $U_J$ now follows easily.

\begin{prop}
Retain the hypotheses and notation of Theorem \ref{thm:coh-exist}.  Then the following squares commute
\begin{equation}
\label{comm:rho-G}
\begin{xy}*!C\xybox{%
\xymatrix{
S^*(\oplus_{\ell = 0}^{r-1} (\fg^{\#(r)}[2p^{r-\ell-1}])) \ar[d]_{\rho_G} \ar[r]^-{\phi_{G,r}} & H^\bu(G_{(r)},k) \ar[d]^{res}\\
S^*(\oplus_{\ell = 0}^{r-2} (\fg^{\#(r-1)}[2p^{r-\ell-2}])) \ar[r]^-{\phi_{G,r-1}} & H^\bu(G_{(r-1)},k). }
}\end{xy}
\end{equation}
\begin{equation}
\label{comm:rho-U}
\begin{xy}*!C\xybox{%
\xymatrix{
S^*(\oplus_{\ell = 0}^{r-1} (\fu_J^{\#(r)}[2p^{r-\ell-1}])) \ar[d]_{\rho_{U_J}} \ar[r]^-{\phi_{U_J,r}} & H^\bu((U_J)_{(r)},k) \ar[d]^{res}\\
S^*(\oplus_{\ell = 0}^{r-2} (\fu_J^{\#(r-1)}[2p^{r-\ell-2}])) \ar[r]^-{\phi_{U_J,r-1}} & H^\bu((U_J)_{(r-1)},k). }
}\end{xy}
\end{equation}
\end{prop}

\begin{proof}
We readily verify that $\phi_{G,r}, \ \phi_{G,r-1}$ provide two faces of a commutative ``cube" with three other faces provided
by (\ref{comm:bracket}), the right square of (\ref{rho-compat}), and the square resulting from the functoriality of 
the restriction map with respect to $G \to GL_N$.    The remaining face is (\ref{comm:rho-G}).

The commutativity of (\ref{comm:rho-U}) is verified similarly, replacing $G \to GL_N$ by $U_J \to G$.
\end{proof}

As summarized in \cite[I.9.9]{J}, the natural map \ $H^*(G,M) \to \varprojlim_r H^*((G_{(r)},M)$ is an isomorphism
for all finite dimensional $G$-modules $M$; a key ingredient in the proof of that isomorphism is the fact that $H^*(G,k)$ 
is  isomorphic to $k$ (in degree 0).   Although $H^*(U_J,k)$ is non-trivial, we do have the following
vanishing result.

\begin{cor}
\label{trivial-limit}
Retain the hypotheses and notation of Theorem \ref{thm:coh-exist}.   Then the sub-algebra 
$$ \varprojlim_r im\{ \phi_{U_J,r}: S^*(\oplus_{\ell = 0}^{r-1} (\fu_J^{\#(r)}[2p^{r-\ell-1}])) \to H^*((U_J)_{(r)},k)  \}$$
of $  \varprojlim_r  H^*((U_J)_{(r)},k) \ \simeq \  H^*(U_J,k)$ is isomorphic to $k$ (in degree 0).
\end{cor}

\begin{proof}
The left exactness of $\varprojlim_r$ implies that 
$$ \varprojlim_r im\{ \phi_{U_J,r}: S^*(\oplus_{\ell = 0}^{r-1} (\fu_J^{\#(r)}[2p^{r-\ell-1}]))  \ \to \ H^*((U_J)_{(r)},k) \} $$
is a subalgebra of $\varprojlim_r  H^*((U_J)_{(r)},k)$ which is isomorphic to $H^*(U_J,k)$ by Theorem \ref{thm:inj}.

By the commutativity of (\ref{comm:bracket}), (\ref{comm:rho-G}), and (\ref{comm:rho-U}),
we observe that no homogeneous element of degree $d, 0 < d < p^s$,  in 
$S^*(\oplus_{\ell = 0}^{r-1} (\fu_J^{\#(r)}[2p^{r-\ell-1}]))$ lies in the image of $\rho^s$ for any $r > s$.
Consequently, there can not exist some non-zero element 
$\{ x_r \} \in  \varprojlim_r im\{ \phi_{U_J,r}: S^*(\oplus_{\ell = 0}^{r-1} (\fu_J^{\#(r)}[2p^{r-\ell-1}]))  
\ \to \ H^*((U_J)_{(r)},k) \} $ of degree $d$:  such an element $\{ x_r \}$ in the inverse limit must satisfy 
$\rho_{r^\prime - r}(x_{r^\prime}) = x_r$ for all $r^\prime \geq r \geq 0$ and would have to
have some $x_r \not= 0, r > s$; then  $\{ x_r \} $ would have to satisfy $x_{r^\prime} \not= 0$ for all $r^\prime \geq r$ (including $r^\prime = r+s$), contradicting the above observation.
\end{proof}

\vskip .1in

As we shall see in the next section, the image of the composition 
$$\varprojlim_r S^*(\oplus_{\ell = 0}^{r-1}((\fu_3/\gamma_2)^{\#(\ell+1)}[2])) \ \to \ \varprojlim_r H^*((U_3/\Gamma_2)_{(r)},k)\ \to \ \varprojlim_r H^*((U_3)_{(r)},k)$$
(which factors through $\varprojlim_r S^*(\oplus_{\ell = 0}^{r-1}(\fu_3^{\#(\ell+1)}[2])) \to \varprojlim_r H^*((U_3)_{(r)},k)$)
is highly non-trivial, where $\bG_a \simeq \Gamma_2  \subset U_3$ is the center of $U_3$.

%%%%%%%%%%%%%%Section1%%%%%%%%%%%%
%%%%%%%%%%%%%%%%%%%%%%%%%%%%%%%
%%%%%%%%%%%%%%%%%%%%%%%%%%%%%%%%%%%%%
%%%%%%%%%%%%%%%%%%%%%%%%%%%%%%%%%%%%%%

\section{The map  $ \ol \eta_{U_J/\Gamma_3,r}: \ol S^*((U_J/\Gamma_3)_{(r)}) \to H^\bu((U_J/\Gamma_3)_{(r)},k)$}
\label{sec:UGamma3}

In this section, we consider groups of the form 
$U_J/\Gamma_3$, where $\Gamma_3 \equiv \Gamma_3(U_J)$ is the third stage of the 
descending central series of the unipotent radical of a parabolic subgroup $P_J \subset G$ of a simple 
algebraic group $G$.  
 Our primary tool is the Lyndon-Hochshield-Serre spectral sequence for a central extension (see
Proposition \ref{prop:UGamma3}) together with the action of the mod-$p$ Steenrod algebra on this 
spectral sequence.  

We recall from \cite{ABS} the description due to H. Azad, M. Barry, and G. Seitz
of the terms of the descending central series of the unipotent radical $U_J$ of $P_J \subset G$  
for some subset $J$ of the set of fundamental positive roots $\Pi$ of a chosen root system for $G$. 
For a positive root $\beta \in \Sigma^+ - \Sigma_J^+$ (where $\Sigma^+$ is the 
set of positive roots  for the root system of $G \supset B \supset T$ and $\Sigma_J^+$ is the set of
positive roots for the root system of $L_J = P_J/U_J \supset T_J$), we adopt the terminology of \cite{ABS}:  
write $\beta = \beta_J + \beta_{J^\prime}$ where $\beta_J$ is a sum $\sum_i c_i \alpha_i$ with each 
$\alpha_i \in J$ and $\beta_{J^\prime}$ is a sum $\sum_j d_j \alpha_j$ with each $\alpha_j \in \Pi-J$;
then the height of $\beta$ is defined to be $\sum_i c_i + \sum_j d_j$, the level of $\beta$ is defined to 
be $\sum_j d_j$ and the shape of $\beta$ is defined to be $\beta_{J^\prime}$.

\begin{prop}  (summary of \S 2 of \cite{ABS})  
\label{prop:ABS} 
Let $G$ be a simple algebraic group of adjoint type, and $P = P_J \subset G$ a parabolic
subgroup, $L_J$ its Levi factor, and $U_J$ its unipotent radical for some subset $J \subset \Pi$.
As usual, assume $p > 2$; for $G$ of type $G_2$, assume $p > 3$.  Consider the descending central series for $U_J$:
 $$\cdots \subset \Gamma_{v+1} = [U_J,\Gamma_v] \subset \cdots \subset 
\Gamma_2 = [U_J,U_J] \subset \Gamma_1 = U_J.$$

For any $v > 1$,  we have the central extension with a natural action of $L_J$:
 \begin{equation}
 \label{central-ext-J}
 1 \to \Gamma_v/\Gamma_{v+1} \to U_J/\Gamma_{v+1} \to U_J/\Gamma_v \to 1.
 \end{equation}
 The commutative group $\Gamma_v/\Gamma_{v+1}$ is a direct product of irreducible $L_J$-modules $V_{\mathcal S}$ indexed by 
 shapes $\mathcal S$ of level $v$; each $V_{\mathcal S}$ is $T$ isomorphic to a product of $U_{-\beta}$
 indexed by $\beta \in \Sigma^+ - \Sigma_J^+$ of shape $\mathcal S$ and level $v$, 
 where $U_{-\beta}$ is the root subgroup with $T$-weight
 $-\beta$; $V_{\mathcal S}$ is a high weight $L_J$-module with highest weight $-\beta_\cS^o$, where $\beta_\cS^o$
 is the unique root of minimal height and shape $\mathcal S$.
 
In particular, if $P_J$ is the minimal parabolic (i.e., equal to the given Borel subgroup $B \subset G$, corresponding to
$J = \emptyset$), then
 $\Gamma_v/\Gamma_{v+1}$ is $T$-isomorphic to  $\prod  U_{-\beta}$ where the product is indexed
 by $\beta \in \Sigma^+$ of level $v$.  
 \end{prop}

 In what follows, we shall denote $Lie(\Gamma_v(U_J))$ by $\gamma_v$.   As in Section \ref{sec:phiU}, $r$ will denote a 
 fixed (but arbitrary) positive integer.
 
 In order to fix notation and $T$-weights for Frobenius twists, we recall
the known computation of  $H^*(\bG_a,k)$ and $H^*(\bG_{a(r)},k)$ (see, for example, \cite[Ch I.4]{J} or \cite{CPSvdK}):  
there is a natural isomorphism of graded commutative $k$-algebras
 \begin{equation}
 \label{Ga-coh}
 S^*(V^{(1)}[2]) \otimes \Lambda^*(V[1]) \ \stackrel{\sim}{\to} \   H^*(\bG_a,k),
 \end{equation}
 where $V = H^1(\bG_a,k)$ is a countable $k$-vector space spanned by $\lambda_1,\lambda_2,\ldots \lambda_s \ldots $ with Frobenius action 
 $F^*(\lambda_s) = \lambda_{s+1}$ , $\Lambda^*(V([1])$ is the exterior algebra on $V$ placed in degree 1, and $S^*(V^{(1)}[2])$ is 
 the polynomial algebra on the Frobenius twist $V^{(1)}$ of $V$ placed in degree 2 spanned by $x_1, 
 x_2, \ldots, x_s, \ldots$, the Bocksteins of the $\lambda_i$.   The action of multiplication by $c \in k$ on $\bG_a$ induces an action on $H^*(\bG_a,k)$ given by
 $c^*(\lambda_i) = c^{p^{i-1}}\lambda_i, \ c^*(x_i) = c^{p^i}x_i$.  This indexing is that of  \cite{CPSvdK} and \cite[Thm 1.3]{SFB2}.
 We recall that the cohomology algebra $H^*(\bG_{a(r)},k)$ 
 of the $r^{th}$ Frobenius kernel $\bG_{a(r)}$ of $\bG_a$ can be identified with
 the quotient of $H^*(\bG_a,k)$ obtained by setting $\lambda_s = 0 = x_s$ for $s > r$.   This finitely generated cohomology algebra admits
 the natural action of the mod-p Steenrod algebra $\mathcal A_p$ (see \cite[1.7]{SFB2} for an explicit description of the action
 of the generators $\cP^i, \beta\cP^i$ of $\mathcal A_p$ on $H^*(\bG_a,k)$).
 
 Observe that $U_J/\Gamma_2$ is a product $\bG_a^{\times s}$ of copies of $\bG_a$, so that its cohomology and that
 of $(\bG_a^{\times s})_{(r)}$ are determined by the above computation and the  K\"unneth theorem.  
 In particular, we conclude that there is a natural injective map
\begin{equation}
\label{eta2}
\eta_{U_J/\Gamma_2,r}: S^*(\oplus_{\ell=0}^{r-1} (\fu_J/\gamma_2)^{\#(\ell+1)}[2]) \ \to \ H^*((U_J/\Gamma_2)_{(r)},k)
\end{equation}
with left inverse $H^*((U_J/\Gamma_2)_{(r)},k) \to S^*(\oplus_{\ell=0}^{r-1} (\fu_J/\gamma_2)^{\#(\ell+1)}[2])$ whose kernel
consists of elements with $p$-th power 0.

We designate $T$-eigenvector generators  for 
$$H^*(U_J/\Gamma_2,k) \ \simeq \ 
S^*(\oplus_{\ell = 0}^\infty (\fu_J/\gamma_2)^{\#(\ell+1)}[2]) \ \otimes \ \Lambda^*(\oplus_{\ell = 0}^\infty (\fu_J/\gamma_2)^{\#(\ell)}[1])$$
by $x_\alpha^{(\ell)}$ of  cohomological degree 2 and $T$-weight  $p^{\ell+1}\alpha$ and $y_\alpha^{(\ell)}$
of cohomological degree 1 and $T$-weight $p^{\ell}\alpha$; here, $\alpha $ ranges
over roots of  $U_J$ of level 1 and $\ell$ is an non-negative integer satisfying $0 \leq \ell$.
Similarly, we designate generators for 
$$H^*(\Gamma_{2}/\Gamma_{3},k) \ \simeq \ 
S^*(\oplus_{\ell = 0}^\infty ((\gamma_{2}/\gamma_{3})^{\#(\ell+1)})[2]) \ \otimes \ 
\Lambda^*(\oplus_{\ell = 0}^\infty (\gamma_2/\gamma_{3})^{\#(\ell)}[1]) $$
by $x_\beta^{(\ell)}$ of cohomological 2 and $y_{\beta}^{(\ell)}$ of cohomological degree 1, with $0 \leq \ell$ and with $\beta$
ranging over $T$-weights of $U_J$ of level $2$.
Consequently, 
\begin{equation}
\label{UGamma3-tensor}
E_2^{*,*}(U_J/\Gamma_3) \ = \ H^*(U_J/\Gamma_2,k) \otimes H^*(\Gamma_2/\Gamma_3,k)
\end{equation}
equals
$$ S^*(\oplus_{\ell = 0}^\infty (\fu_J/\gamma_3)^{\#(\ell+1)}[2]) \ 
\otimes \ \Lambda^*(\oplus_{\ell = 0}^\infty (\fu_J/\gamma_3)^{\#(\ell)}[1]),$$
where $E_2^{*,*}(U_J/\Gamma_3)$ is the spectral sequence considered in the following proposition.

The indexing we adopt (for example, in Proposition \ref{prop:UGamma3}) relates to the above indexing 
as follows for cohomology classes of $\bG_a$: $\lambda_i$ corresponds to $y_\alpha^{(\ell)}$ and 
$x_i$ corresponds to $x_\alpha^{(\ell)}$  with $\ell = i-1$.  
The summation $\sum_{\alpha + \alpha^\prime = \beta}$
indicates a sum of pairs of roots $\alpha, \alpha^\prime$ with the property that $[X_\alpha,X_{\alpha^\prime}] = X_\beta
\in \fu_J/\gamma_{v+1}$.

\begin{prop}
\label{prop:UGamma3}
Retain the notation and hypotheses of Proposition \ref{prop:ABS}.  
Consider the $T$-equivariant Lyndon-Hochshild-Serre spectral sequence 
\cite{HS} for the extension $1\to \Gamma_2/\Gamma_3 \to U_J/\Gamma_3 \to U_J/\Gamma_2 \to 1$: 
\begin{equation}
\label{specseqGamma3}
E_2^{a,b}(U_J/\Gamma_3) \  = \ H^a(U_J/\Gamma_2,k) \otimes H^b(\Gamma_2/\Gamma_3,k) \ 
\Rightarrow H^{a+b}(U_J/\Gamma_3,k).
\end{equation}

For any $\ell \geq 0$, $j \geq 0$, and any $\beta$ a weight of level 2:
\begin{enumerate}
\item
$$d_2^{0,1}(y_\beta^{(\ell)}) \ = \  \sum_{\alpha + \alpha^\prime = \beta} 
y_\alpha^{(\ell)} \wedge y_{\alpha^\prime}^{(\ell)} \in H^2(U_J/\Gamma_2,k).$$ 
\item
$ d_{2p^j+1}^{0,2p^j}((x_{\beta}^{(\ell)})^{p^j}) \ = \  \sum_{\alpha + \alpha^\prime = \beta} 
 \{ (x_{\alpha}^{(\ell)})^{p^j} \otimes 
y_{\alpha^\prime}^{(\ell+1+j)} - (x_{\alpha^\prime}^{(\ell)})^{p^j} \otimes y_{\alpha}^{(\ell+1+j)}\}$
is non-zero in $H^{2p^{j+1}+1}(U_J/\Gamma_2,k)$.  Thus, $(x_{\beta}^{(\ell)})^{p^j} \in H^{2p^j}(\Gamma_2/\Gamma_3,k)$
does not lie in the image of
$H^\bu(U_J/\Gamma_3,k)$.
\item
$\beta\cP^{p^{j}}(\sum_{\alpha + \alpha^\prime = \beta}    \{ (x_{\alpha}^{(\ell)})^{p^j} \otimes 
y_{\alpha^\prime}^{(\ell+1+j)} - (x_{\alpha^\prime}^{(\ell)})^{p^j} \otimes y_{\alpha}^{(\ell+1+j)}\})$ equals
\begin{equation}
\label{eqn:relnU3}
 \sum_{\alpha + \alpha^\prime = \beta}   \{ (x_{\alpha}^{(\ell)})^{p^{j+1}} \otimes 
x_{\alpha^\prime}^{(\ell+1+j)} - (x_{\alpha^\prime}^{(\ell)})^{p^{j+1}} \otimes x_{\alpha}^{(\ell+1+j)}\},
\end{equation}
non-zero in $H^{2p^{j+1}+2}(U_J/\Gamma_2,k)$.  The expression 
(\ref{eqn:relnU3}) maps to 0 in $H^{2p^{j+1}+2}(U_J/\Gamma_3,k)$.
\end{enumerate}
\end{prop}

\begin{proof}
Assume first that $U_J$ equals $U_3$ and consider the extensions
$$ (\bG_a)_{(i)}/ (\bG_a)_{(i-1)} \ \to \ (U_3)_{(i)}/ (U_3)_{(i-1)} \  \to \ 
(\bG_a^{\times 2})_{(i)}/ (\bG_a^{\times 2})_{(i-1)}.$$
Since  $(\bG_a)_{(i)}/ (\bG_a)_{(i-1)}$ lies in the commutator of $(U_3)_{(i)}/(U_3)_{(i-1)}$,
we conclude that the map
$H^1((U_3)_{(i)}/ (U_3)_{(i-1)},k) \to H^1((\bG_a)_{(i)}/ (\bG_a)_{(i-1)},k)$ is 0,
so that the differential $d_2^{0,1}$ must be non-zero for each of these extensions.  We thus
conclude using induction on $r$ and the fact that differentials commute with the
action of $T_3$ that $d_2^{0,1}(y_\beta^{(\ell)})$ is  a sum of the form given in (1)
with non-zero coefficients of each summand $y_\alpha^{(\ell)} \wedge y_{\alpha^\prime}^{(\ell)}$.
These coefficients must be equal for varying $\ell$ using functoriality with respect to 
Frobenius maps $U_{3(i)} \to U_{3(i+1)}$.  Defining our group schemes over $\bZ$,
we conclude these coefficients must be $\pm 1$.   We have chosen the ordering of the pairs $\alpha,\alpha^\prime$
so that these coefficients are all $+1$.

For a more general $U_J/\Gamma_3$, we consider a pair $\alpha, \alpha^\prime$ with 
$\alpha + \alpha^\prime = \beta$ and define the subgroup $R_{\alpha,\alpha^\prime}
\subset U_J/\Gamma_3$ to be the subgroup generated by the root subgroups $U_\alpha, U_{\alpha^\prime}$.
Thus, $R_{\alpha,\alpha^\prime} \simeq U_3$ with center $U_\beta$.  The functoriality of the LHS spectral
sequence implies that 
\begin{equation}
\label{square:d2}
\begin{xy}*!C\xybox{%
\xymatrix{
H^1(\Gamma_2/\Gamma_3,k) \ar[r] \ar[d]^{d_2^{0,1}} & H^1(U_\beta,k) \ar[d]^{d_2^{0,1}} \\
H^2(U_J/\Gamma_2,k) \ar[r] & H^2(U_\alpha\times U_{\alpha^\prime},k) 
}
}\end{xy}
\end{equation}
commutes, so that $d_2^{0,1}(y_\beta^{(\ell)})$ must be given by (1) plus additional terms.   Yet there
are no other eigenvectors of $T$-weight $p^\ell\beta$ in $H^2(U_J/\Gamma_2,k)$, so there can be no additional terms
in the formula for $d_2^{0,1}(y_\beta^{(\ell)})$ .

For $j = 0$, (2) follows from the equality $x_{\beta}^{(\ell)} = (\beta\cP^0)(y_\beta^{(\ell)})$, the fact that
$\beta\cP^0$ commutes with transgression (\cite{Kudo}), and the Cartan formula
$$\beta \cP^j(u\cdot v) \ = \  \sum_{i=0}^j ((\beta \cP^i(u)\cdot \otimes \cP^{j-i} (v) - (\cP^i)(u) \cdot (\beta \cP^{j-i}))(v)).$$
In particular, this tells us that 
$$\beta\cP^0(y^{(\ell)}_\alpha \wedge y_{\alpha^\prime}^{(\ell)}) = \beta\cP^0(y_\alpha^{(\ell)}) \otimes \cP^0(y_{\alpha^\prime}^{(\ell)}) 
- \cP^0(y_\alpha^{(\ell)}) \otimes \beta \cP^0(y_{\alpha^\prime}^{(\ell)}) =
x_\alpha^{(\ell)} \otimes y_{\alpha^\prime}^{(\ell+1)} - x_{\alpha^\prime}^{(\ell)} \otimes y_{\alpha}^{(\ell+1)}.$$
To prove (2) for $j > 0$,  we recall that $\cP^{p^i}$ applied to $(x_{\beta}^{(\ell)})^{p^i}$ equals $(x_{\beta}^{(\ell)})^{p^{i+1}}$.
Using the fact that the Steenrod action commutes with the differentials in the spectral sequence and repeated applications 
of the Cartan formula, we verify (2) by computing
$ d_{2p^j+1}^{0,2p^j}((x_{\beta}^{(\ell)})^{p^j})$, the result of applying $d_{2p^j+1}^{0,2p^j}$ to
$(\cP^{p^{j-1}} \circ \cdots \cP^1 \circ \beta\cP^0)(y_\beta^{(\ell)})$.
The fact that $ d_{2p^j+1}^{0,2p^j}((x_{\beta}^{(\ell)})^{p^j}) \not= 0$ follows 
from the explicit computation of $H^\bu(U_J/\Gamma_2,k)$.
Because some differential in the spectral sequence is non-vanishing on $(x_\beta^{(\ell)})^{p^j}$,
it  does not lie in the image of $H^\bu(U_J/\Gamma_3,k)$.

The computation of assertion (3) follows from the Cartan formula for $\beta \cP^{p^j}$
and the detailed description of $\cP^i$ and $\beta\cP^i$
given in \cite[1.7]{SFB2}.  The non-vanishing of (\ref{eqn:relnU3}) follows once again from the
explicit computation of $H^\bu(U_J/\Gamma_2,k)$.
\end{proof}

The restriction map for the embedding $(U_J/\Gamma_3)_{(r)} \to U_J/\Gamma_3$ determines a map
from the spectral sequence (\ref{specseqGamma3}) to the spectral sequence
\begin{equation}
\label{specseqGamma3r}
E_2^{a,b}((U_J/\Gamma_3)_{(r)}) \  = \ H^a((U_J/\Gamma_2)_{(r)},k) \otimes 
H^b((\Gamma_2/\Gamma_3)_{(r)},k) \ 
\Rightarrow H^{a+b}((U_J/\Gamma_3)_{(r)},k).
\end{equation}
considered in the next proposition.  On $E_2$-terms, this map sends $y_\beta^{(\ell)}, x_\beta^{(\ell)}, 
y_\alpha^{(\ell)}, x_\alpha^{(\ell)}$ to 0 for $\ell \geq r$.  

To exhibit the action of $T$, we retain the indexing of Proposition \ref{prop:UGamma3}, viewing  $E_2^{*,*}((U_J/\Gamma_3)_{(r)})$ 
as the tensor 
product 
\begin{equation}
\label{E**}
S^*(\oplus_{\ell = 0}^{r-1} (\fu_J/\gamma_2)^{\#(\ell+1)}[2]) \ \otimes 
\ \Lambda^*(\oplus_{\ell = 0}^{r-1}  (\fu_J/\gamma_2)^{\#(\ell)}[1]) \ \otimes 
\end{equation}
$$S^*(\oplus_{\ell = 0}^{r-1}  ((\gamma_{2}/\gamma_{3})^{\#(\ell+1)})[2]) \ \otimes \ 
\Lambda^*(\oplus_{\ell = 0}^{r-1} (\gamma_2/\gamma_{3})^{\#(\ell)}[1]).$$

\begin{prop}
\label{prop:UGamma3r}
Retain the notation and hypotheses of Proposition \ref{prop:ABS}, and consider 
the spectral sequence (\ref{specseqGamma3r}) for the central extension
\begin{equation}
\label{specseq:Gamma3r}
1 \ \to \ (\Gamma_2/\Gamma_3)_{(r)} \ \to \ (U_J/\Gamma_3)_{(r)} \ \to \ (U_J/\Gamma_2)_{(r)} \ \to \ 0.
\end{equation}
For any $\beta$ of level 2, 
\begin{enumerate}
\item
$(x_\beta^{(\ell)})^{p^j} \in S^*(\oplus_{\ell = 0}^{r-1} (\gamma_2/\gamma_3)^{\#(\ell+1)}[2]) \subset 
E_2^{0,*}((U_J/\Gamma_3)_{(r)})$ is 
a permanent cycle if and only if $ \ell+1+j \geq r$.
\item
For any $\ell , j \geq 0$ with $\ell+1+j < r$, 
\begin{equation}
\label{reln:S2}
\sum_{\alpha + \alpha^\prime = \beta,\alpha < \alpha^\prime}    \{ (x_{\alpha}^{(\ell)})^{p^{j+1}} \otimes 
x_{\alpha^\prime}^{(\ell+1+j)} - (x_{\alpha^\prime}^{(\ell)})^{p^{j+1}} \otimes x_{\alpha}^{(\ell+1+j)}\} 
\ = \ 0 \ \in H^{2p^{j+1}+2}((U_J/\Gamma_3)_{(r)},k).
\end{equation}
\item
For $U_J = U_3$ (with $\Gamma_3 = 1$), the $p^{r-\ell - j -1}$-st power of relation 
(\ref{reln:S2}) in $H^\bu((U_3)_{(r)},k)$ is the restriction to $(U_3)_{(r)}$
of the relation $X^{1,2}(\ell)\cdot X^{2,3}(\ell^\prime) - X^{2,3}(\ell) \cdot X^{1,2}(\ell^\prime)$  of 
Theorem \ref{thm:SFB-factor}.
\end{enumerate}
\end{prop}

\begin{proof}
The vanishing of $y_\alpha^{(\ell)}, \ell \geq r$ together with Proposition 
\ref{prop:UGamma3}(2) immediately implies that $(x_\beta^{(\ell)})^{p^j}$ is a permanent cycle
if $ \ell+1+j \geq r$.  Conversely, if $ \ell+1+j < r$, then Proposition \ref{prop:UGamma3} tells us that
$d_{2p^j}^{0,2p^j}$ does not vanish on $(x_\beta^{(\ell)})^{p^j}$.

To conclude (2), we first recall that Proposition \ref{prop:UGamma3}(2) implies that 
\begin{equation}
\label{bvalue}
\sum_{\alpha + \alpha^\prime = \beta}    \{ (x_{\alpha}^{(\ell)})^{p^j} \otimes 
y_{\alpha^\prime}^{(\ell+1+j)} - (x_{\alpha^\prime}^{(\ell)})^{p^j} \otimes y_{\alpha}^{(\ell+1+j)}\} 
 \ \in \ H^{2p^{j+1}}(U_J/\Gamma_2)_{(r)},k)
\end{equation}
is non-zero and in the image of a differential in the spectral sequence (\ref{specseqGamma3r}).  This implies
that (\ref{bvalue}) lies in the kernel of the inflation map.  One obtains the relation of (\ref{reln:S2}) by applying
$\beta \cP^j$ to (\ref{bvalue}) to obtain a class in $H^{2p^{j+1}+2}((U_J/\Gamma_2)_{(r)},k)$ whose inflation
gives (\ref{reln:S2}); since $\beta \cP^j$ commutes with the inflation map, we conclude the asserted vanishing.

Assertion (2) follows from Proposition \ref{prop:UGamma3}(2), since 
$\sum_{\alpha + \alpha^\prime = \beta}    \{ (x_{\alpha}^{(\ell)})^{p^{j+1}} \otimes 
y_{\alpha^\prime}^{(\ell+1+j)} - (x_{\alpha^\prime}^{(\ell)})^{p^{j+1}} \otimes y_{\alpha}^{(\ell+1+j)}\} $ is
a boundary and the restriction map commutes with the Bockstein.

The fact that the $p^{r-\ell -j-1}$-st power of (\ref{reln:S2}) equals the relation
$X^{1,2}(\ell)\cdot X^{2,3}(\ell^\prime) - X^{2,3}(\ell) \cdot X^{1,2}(\ell^\prime)$  of 
Theorem \ref{thm:SFB-factor} is immediate from the identification of $X^{1,2}(\ell)$ with $(x_{\alpha}^{(\ell)})^{p^{r-\ell-1}}$
and $X^{2,3}(\ell)$ with $(x_{\alpha^\prime}^{(\ell)})^{p^{r-\ell-1}}$.
\end{proof}

We view $ S^*(\oplus_{\ell=0}^{r-1} (\gamma_2/\gamma_3)^{\#(\ell+1)}[2])$ as a subalgebra of 
$E_2^{0,*}((U_J/\Gamma_3)_{(r)})$ using the identification (\ref{E**}).   Proposition \ref{prop:UGamma3r}(1)
tells us that the subalgebra
$$S^*(\oplus_{\ell=0}^{r-1} (\gamma_2/\gamma_3)^{\#(r)}[2p^{r-\ell-1}]) \ \subset \ 
S^*(\oplus_{\ell=0}^{r-1} (\gamma_2/\gamma_3)^{\#(\ell+1)}[2])$$
(defined as the image of the endomorphism $S^*(\oplus_{\ell=0}^{r-1}F^{r-\ell-1})$ on $S^*(\oplus_{\ell=0}^{r-1} (\gamma_2/\gamma_3)^{\#(\ell+1)}[2])$)
consists of permanent cycles in the spectral sequence (\ref{specseqGamma3r}).  
The following corollary tells us that this subalgebra is the intersection of the permanent cycles in 
$E_2^{*,*}((U_J/\Gamma_3)_{(r)})$ with $S^*(\oplus_{\ell=0}^{r-1} (\gamma_2/\gamma_3)^{\#(\ell+1)}[2])$.
What this corollary does {\it not} do is identify all permanent cycles of $E_2^{*,*}((U_J/\Gamma_3)_{(r)})$.

\begin{cor}
\label{cor:permanent3}
If $z \in S^*(\oplus_{\ell=0}^{r-1} (\gamma_2/\gamma_3)^{\#(\ell+1)}[2])$ does not lie in the subalgebra \\
$S^*(\oplus_{\ell=0}^{r-1} (\gamma_2/\gamma_3)^{\#(r)}[2p^{r-\ell-1}])$,  then there exists some 
differential of (\ref{specseqGamma3r})
which is non-zero on $z$.
\end{cor}

\begin{proof}
We employ the fact that the differentials in the spectral sequence are $k$-linear derivations.  Let
$\{ \beta_i, i\in I\}$ be the set of positive roots of $U_J$ of level 2.  Consider a monomial
$w = \prod_{i\in I} \prod_{\ell=0}^{r-1} (x_{\beta_i}^{(\ell)})^{n_{i,\ell}}$ with some $n_{i,\ell}$ not divisible by $p^{r-\ell-1}$.  
Let $p^j$ be the smallest power of $p$ such that $p^j$ divides some $n_{i,\ell}$, and
$p^{j+1}$ does not divide $n_{i,\ell}$.  
Then $d_{2p^j+1}^{0,\sum_\ell 2n_{i,\ell}}(w)$ is a sum of non-zero terms indexed by those $i,\ell$ with 
$p^j$ but not $p^{j+1}$ dividing $n_{i,\ell}$, each summand having a different $T$-weight.  Thus, $w$ is not a permanent cycle.

More generally, different such monomials have different $T$-weights, so that no non-trivial sum
of such monomials is a permanent cycle.
\end{proof}

\begin{note}
\label{note:compare}
Let $U_N \subset GL_N$ denote the subgroup of strictly upper triangular matrices.
The generators $X^{i,j}(\ell) \in S^*((\fu_N/\gamma_3)^{\#(r)}[2p^{r-\ell-1}])$ (as discussed following Proposition \ref{prop:basic})
correspond to $(x_{i,j}^{(\ell)})^{p^{r-\ell-1}} \in E_2^{*,*}((U_N/\Gamma_3)_{(r)})$ of Proposition \ref{prop:UGamma3r}.

We extend this notation, denoting generators of $S^*((\fu_J/\gamma_3)^{\#(r)}[2p^{r-\ell-1}])$ by $X^\beta(\ell) \in
S^*((\fu_J/\gamma_3)^{\#(r)}[2p^{r-\ell-1}])$.
These classes have the same weight and cohomological degree as the classes $(x_\beta^{(\ell)})^{p^{r-\ell-1}} \in 
E_2^{*,*}((U_J/\Gamma_3)_{(r)})$.   We shall see that the representative in $E_2^{*,*}((U_J/\Gamma_3)_{(r)})$
of $\eta_{U_J/\Gamma_3,r}(X^\beta(\ell)) \in H^*((U_J/\Gamma_3)_{(r)},k)$ is $(x_\beta^{(\ell)})^{p^{r-\ell-1}}$.
\end{note}

The uniqueness given in the following proposition enables us to specify the map $\eta_{U_J/\Gamma_3,r}$.
We are particularly interested in the special case $U_J/\Gamma_3 = U_3$.

\begin{prop}
\label{prop:1-dim}
Retain the notation and hypotheses of Proposition \ref{prop:ABS}.    Assume each root $\beta$ of $U_J$
of level 2 can be written uniquely as a sum $\alpha + \alpha^\prime$ of roots $U_J$ of level 1.
Then  there exists a unique $T$-equivariant $k$-linear map
$$\eta_r: (\fu_J/\gamma_3)^{\#(r)}[2p^{r-1}] \ \to \ H^{2p^{r-1}}((U_J/\Gamma_3)_{(r)},k)$$
which fits in the following commutative diagram
\begin{equation}
\label{diagram:eta}
\begin{xy}*!C\xybox{%
\xymatrix{
(\fu_J/\gamma_2)^{\#(r)}[2p^{r-1}] \ar[r] \ar[d] & (\fu_J/\gamma_3)^{\#(r)}[2p^{r-1}] \ar[r] \ar[d]^{\eta_r} 
& (\gamma_2/\gamma_3)^{\#(r)}[2p^{r-1}] \ar[d]\\
H^{2p^{r-1}}((U_J/\Gamma_2)_{(r)},k) \ar[r] & H^{2p^{r-1}}((U_J/\Gamma_3)_{(r)},k) \ar[r] &
H^{2p^{r-1}}((\Gamma_2/\Gamma_3)_{(r)},k).
}
}\end{xy}
\end{equation}
Here, the left and right vertical maps are given by the inclusions $S^*((\fu_J/\gamma_2)^{\#(1)}[2]) \to 
H^\bu((U_J/\Gamma_2)_{(r)},k)$ and $S^*((\gamma_2/\gamma_3)^{\#(1)}[2]) \to 
H^\bu((\Gamma_2/\Gamma_3)_{(r)},k)$, the upper horizontal maps are 
the evident ones, the lower horizontal maps are those given by functoriality.

Furthermore, the map $\eta_r$ fits in a commutative square
\begin{equation}
\label{diagram:compat-eta}
\begin{xy}*!C\xybox{%
\xymatrix{
(\fu_J/\gamma_3)^{\#(r)}[2p^{r-1}] \ar[r]^{\eta_r} \ar[d] & H^{2p^{r-1}}((U_J/\Gamma_3)_{(r)},k) \ar[d] \\
(\fu_J)^{\#(r)}[2p^{r-1}] \ar[r]_{(\phi_{U_J,r})_|} & H^{2p^{r-1}}((U_J)_{(r)},k)  
}
}\end{xy}
\end{equation}
whose vertical maps are inflation maps and whose lower horizontal map  is the restriction of
the map $\phi_{U,r}$ of Theorem \ref{thm:coh-exist} with $U = U_J$ satisfying the hypotheses of this 
proposition.
\end{prop}

\begin{proof}
The existence of some  $\eta$ fitting in the commutative diagram (\ref{diagram:eta}) is 
implied by Proposition \ref{prop:UGamma3r}(1).
To prove the uniqueness of $\eta$, it suffices to verify for each root $\beta$ of $U_J$ of level 2 that the
$T$-weight space of $H^{2p^{r-1}}((U_J/\Gamma_3)_{(r)},k)$ of weight $p^r\beta$ is 1-dimensional.  This would imply
the uniqueness of the choice of  class $\eta(X^\beta(0)) \in H^{2p^{r-1}}((U_J/\Gamma_3)_{(r)},k)$ 
fitting in the commutative diagram (\ref{diagram:eta}).

We search in ${}^{AJ}E_1^{*,*}$ as given in (\ref{summands}) for $T$-weight vectors with $T$-weight $p^r\beta$ and
cohomology degree $2p^{r-1}$ other than $(x_\beta^{(0)})^{p^{r-1}} \in E_2^{0,2p^{r-1}}$.   Consider a simple
tensor of the specified weight and degree, in other words a monomial $z$ in $x$'s and $y$'s.  Because the 
weight of $y_\alpha^{(0)} \wedge y_{\alpha^\prime}^{(0)}$ is not divisible by $p$, this does not divide the monomial $z$.
None of the factors of the monomial $z$ of weight $p^r\beta$ and degree $2p^{r-1}$ can be of the form 
$x^{(\ell)}$ for $\ell > 1$ or for $y^{(\ell)}$ for $\ell > 2$ because such a factor would increase the weight too ``fast" with respect to 
increase of the resulting degree by either 2 or 1.  

Thus, the only allowable weight vectors of cohomology degree $2p^{r-1}$ are scalar multiples of $(x_\beta^{(0)})^{p^{r-1}}$
and $(x_\beta^{(0)})^{p^{r-1}-1} \otimes y_\alpha^{(1)} \wedge y_{\alpha^\prime}^{(1)}$. 
 Since $y_\beta^{(i)} \in E_2^{0,1}$ is a permanent cycle, we conclude that $d_2((x_\beta)^{(i)}) = 0$ (using the 
 fact that differentials in the spectral sequence commute with Bocksteins).    Consequently,
 the value of  the derivation $d_2^{0,2p^{r-1}-1}$ applied to 
$(x_\beta^{(0)})^{p^{r-1}-1} \otimes y_\beta^{(1)}$ equals   $(x_\beta^{(0)})^{p^{r-1}-1} \otimes y_{\alpha_1}^{(1)} \wedge 
y_{\alpha_2}^{(1)}$.  We conclude that the class of $(x_\beta^{(0)})^{p^{r-1}} \in E_\infty^{0,2p^{r-1}}$ spans 
the $p^r\beta$ weight space of degree $2p^{r-1}$ of $E_\infty^{0,2p^{r-1}}$.  This implies that 
the $p^r\beta$ weight space of 
$H^{2p^{r-1}}((U_J/\Gamma_3)_{(r)},k)$ is 1-dimensional. 

Finally, the commutativity of (\ref{diagram:compat-eta}) follows from the $T$-equivariance of $\eta$ and the uniqueness
assertion of Theorem \ref{thm:coh-exist}.
\end{proof}

\begin{defn}
\label{defn:S*}
Retain the notation and hypotheses of Propositions \ref{prop:1-dim}.  We define 
$S^*((U_J/\Gamma_v)_{(r)})$ to be
\begin{equation}
S^*(\oplus_{\ell=0}^{r-1} (\fu_J/\gamma_2)^{\#(\ell+1)}[2]) 
\otimes_{S^*(\oplus_{\ell=0}^{r-1} (\fu_J/\gamma_2)^{\#(r)}[2p^{r-\ell-1}])} 
S^*(\oplus_{\ell=0}^{r-1} (\fu_J/\gamma_v)^{\#(r)}[2p^{r-\ell-1}]).
\end{equation} 
In other words, $S^*((U_J/\Gamma_v)_{(r)})$ is the coproduct in the category of
commutative $k$ algebras of $S^*(\oplus_{\ell=0}^{r-1} (\fu_J/\gamma_2)^{\#(\ell+1)}[2]) $
and $S^*(\oplus_{\ell=0}^{r-1} (\fu_J/\gamma_v)^{\#(r)}[2p^{r-\ell-1}])$ over 
$S^*(\oplus_{\ell=0}^{r-1} (\fu_J/\gamma_2)^{\#(r)}[2p^{r-\ell-1}])$.  
The $T$-equivariant splitting $\fu_J/\gamma_v \ \simeq \ (\fu_J/\gamma_2) \oplus (\gamma_2/\gamma_v)$
gives the $T$-equivariant splitting
\begin{equation}
\label{Ssplit}
S^*((U_J/\Gamma_v)_{(r)}) \ \simeq \ S^*(\oplus_{\ell=0}^{r-1} (\fu_J/\gamma_2)^{\#(\ell+1)}[2]) \otimes
S^*(\oplus_{\ell=0}^{r-1} (\gamma_2/\gamma_v)^{\#(r)}[2p^{r-\ell-1}]).
\end{equation}
\end{defn}

\begin{defn}
\label{defn:eta}
Retain the notation and hypotheses of Propositions \ref{prop:ABS}.   We define 
\begin{equation}
\label{tildeg}
\tilde g: S^*(\oplus_{\ell=0}^{r-1} (\gamma_2/\gamma_3)^{\#(r)}[2p^{r-\ell-1}]) \ \to \ H^\bu((U_J/\Gamma_3)_{(r)},k)
\end{equation}
to be the map of $k$-algebras determined by  the $\ell$-th Frobenius twists \\
$\eta_{r-\ell}^{(\ell)}: (\fu_J/\gamma_3)^{\#(r)}[2p^{r-\ell-1}] \to H^{2p^{r-\ell-1}}((U_J/\Gamma_3)_{(r)},k)$
of the maps \\
$\eta_{r-\ell}:  (\fu_J/\gamma_3)^{\#(r-\ell)}[2p^{r-\ell-1}] \to H^{2p^{r-\ell-1}}((U_J/\Gamma_3)_{(r-\ell)},k)$
constructed in Proposition \ref{prop:1-dim} (with $r$ replaced by $r-\ell$).   

We define 
\begin{equation}
\label{tildef}
\tilde f: S^*(\oplus_{\ell=0}^{r-1} (\fu_J/\gamma_2)^{\#(\ell+1)}[2]) \ \to \ H^\bu((U_J/\Gamma_2)_{(r)},k) \ \to \ 
H^\bu((U_J/\Gamma_3)_{(r)},k) 
\end{equation} 
to be the composition of $\eta_{U_J/\Gamma_2,r}$ given in (\ref{eta2}) and the inflation map.

We define
\begin{equation}
\label{formula:eta}
\eta_{U_J/\Gamma_3,r} = \tilde f \odot \tilde g: S^*((U_J/\Gamma_3)_{(r)}) \ \to \ H^\bu((U_J/\Gamma_3)_{(r)},k)
\end{equation}
to be the coproduct of $\tilde g$ and the map $\tilde f$ .
The maps $\tilde g$ and $\tilde f$  agree on $S^*((\oplus_{\ell=0}^{r-1} (\fu_J/\gamma_2)^{\#(r)}[2p^{r-\ell-1}])$ by 
commutativity of the left square of (\ref{diagram:eta}), so that $\eta_{U_J/\Gamma_3,r}$ is well defined.
\end{defn}

\vskip .1in

We give $S^*((U_J/\Gamma_3)_{(r)})$ the decreasing filtration whose subalgebra of level $2i$ is the coproduct of
$S^{\geq i}(\oplus_{\ell=0}^{r-1} (\fu_J/\gamma_2)^{\#(\ell+1)}[2])$ and 
$S^*((\oplus_{\ell=0}^{r-1} (\gamma_2/\gamma_3)^{\#(r)}[2p^{r-\ell-1}])$ over 
$S^{\geq i}(\oplus_{\ell=0}^{r-1} (\fu_J/\gamma_2)^{\#(r)}[2p^{r-\ell-1}])$.

\begin{prop}
\label{prop:eta3}
The map $\eta_{U_J/\Gamma_3,r}$ is a naturally defined 
$T$-equivariant map of filtered $k$-algebras, where we give 
$H^\bu((U_J/\Gamma_3)_{(r)},k)$ the LHS filtration of Proposition \ref{prop:UGamma3r} and  $S^*((U_J/\Gamma_3)_{(r)})$
the filtration described immediately above.

Furthermore,
\begin{equation}
\label{map:grS*}
gr(\eta_{U_J/\Gamma_3,r}) = f \odot  g: S^*((U_J/\Gamma_3)_{(r)}) \ \to \ gr\{ H^\bu((U_J/\Gamma_3)_{(r)},k)\}
\end{equation}
where $g: S^*((\oplus_{\ell=0}^{r-1} (\gamma_2/\gamma_3)^{\#(r)}[2p^{r-\ell-1}]) \to E_\infty^{0,*}$ is defined to be the 
composition of
$\tilde g$ with the natural map $H^\bu((U_J/\Gamma_3)_{(r)},k) \to E_\infty^{0,*}$ and $f: S^*((U_J/\Gamma_2)_{(r)})
\to H^\bu((U_J/\Gamma_2)_{(r)},k) \to E_\infty^{*,0}$ is the map whose composition with the natural inclusion
$E_\infty^{*,0} \to H^*((U_J/\Gamma_3)_{(r)},k)$ equals $\tilde f$.
\end{prop}

\begin{proof}
The map $\tilde f$ is induced by $U_J/\Gamma_3 \to U_J/\Gamma_2$ and thus is filtration preserving. 
Observe that $1 \otimes S^*(\oplus_{\ell = 0}^{r-1}(\gamma_2/\gamma_3)^{\#(r)}[2p^{r-\ell-1}]) 
\subset S^*((U_J/\Gamma_3)_{(r)})$ has filtration degree 0 as does $E_\infty^{0,*}$, so that $\tilde g$ 
is also a map of filtered algebras.  The mutliplicative property of these filtrations thus implies that
$\eta_{U_J/\Gamma_3,r}$ itself is a map of filtered algebras.

The identification of $gr\{\eta_{U_J/\Gamma_3,r}\}$ is verified by proving the commutativity of the following
two diagrams
\begin{equation}
\label{U3-short1}
\begin{xy}*!C\xybox{%
\xymatrix{ S^*(\oplus_{\ell = 0}^{r-1}(\fu_J/\gamma_2)^{\#(\ell+1)}[2]) \ar[r]^= \ar[d]^-{\eta_{U_J/\Gamma_2,r}} 
 & S^*(\oplus_{\ell = 0}^{r-1}(\fu_J/\gamma_2)^{\#(\ell+1)}[2]) \ar[d] \ar[r]  &   
S^*((U_J/\Gamma_3)_{(r)}) \ar[d]^-{ \eta_{U_J/\Gamma_3,r} }\\
H^\bu((U_J/\Gamma_2)_{(r)},k) \ar[r] & E_\infty^{*,0}((U_J/\Gamma_3)_{(r)}) \ar[r] & H^*((U_J/\Gamma_3)_{(r)},k) , 
 }
}\end{xy}
\end{equation}
 \begin{equation}
\label{U3-short2}
\begin{xy}*!C\xybox{%
\xymatrix{  
S^*((U_J/\Gamma_3)_{(r)}) \ar[d]^-{\eta_{U_J/\Gamma_3,r}} \ar[r] &  
S^*(\oplus_{\ell = 0}^{r-1}(\gamma_2/\gamma_3)^{\#(r)}[2p^{r-\ell-1}]) \ar[r] 
\ar[d]^{g} & S^*(\oplus_{\ell = 0}^{r-1}(\gamma_2/\gamma_3)^{\#(\ell+1)}[2]) \ar[d]^{\tilde g} \\
H^\bu((U_J/\Gamma_3)_{(r)},k)  \ar[r] & E_\infty^{0,*}((U_J/\Gamma_3)_{(r)}) \ar[r] & H^*((\Gamma_2/\Gamma_3)_{(r)},k)
}
}\end{xy}
\end{equation}
 By Proposition \ref{prop:UGamma3r}(2), the left square of  (\ref{U3-short1}) commutes.  By definition of 
$\eta_{U_J/\Gamma_3,r}$, the right square of  (\ref{U3-short1}) also commutes. 

The commutativity of the outer square of (\ref{U3-short2}) arises from the naturality of the restriction maps for
$\Gamma_2/\Gamma_3 \to U_J/\Gamma_3$.  The  commutativity of the two squares of (\ref{U3-short2})  
thus follows from the fact 
 that $S^*(\oplus_{\ell = 0}^{r-1}(\gamma_2/\gamma_3)^{\#(r)}[2p^{r-\ell-1}])$ is the image of the upper composition of
(\ref{U3-short2}) and the fact that $E_\infty^{0,*}((U_J/\Gamma_3)_{(r)})$ is the image of the restriction map 
$H^\bu((U_J/\Gamma_3)_{(r)},k)
\to H^\bu((\Gamma_2/\Gamma_3)_{(r)},k)$ by a standard property of Grothendieck spectral sequences.
\end{proof}

\begin{defn}
\label{defn:Q2}
We define $Q((U_J/\Gamma_2)_{(r)})$ to be the quotient of $S^*(\oplus_{\ell=0}^{r-1} (\fu_J/\gamma_2)^{\#(\ell+1)}[2])$
by the ideal generated by the elements of (\ref{reln:S2}), which we denote by $J_2$:
$$Q((U/\Gamma_2)_{(r)}) \ \equiv \ S^*(\oplus_{\ell=0}^{r-1} (\fu_J/\gamma_2)^{\#(\ell-1)}[2])/J_2.$$

We define $\ol S^*((U_J/\Gamma_3)_{(r)})$ to be the tensor product of $S^*((U_J/\Gamma_3)_{(r)})$ and
$Q((U_J/\Gamma_2)_{(r)})$ over  $S^*(\oplus_{\ell=0}^{r-1} (\fu_J/\gamma_2)^{\#(\ell+1)}[2])$.

The $T$-equivariant splitting $\fu_J/\gamma_3 \ \simeq \ (\fu_J/\gamma_2) \oplus (\gamma_2/\gamma_3)$
gives the $T$ equivariant splitting 
\begin{equation}
\label{Sbar-split}
\ol S^*((U_J/\Gamma_3)_{(r)}) \ \simeq \ Q((U_J/\Gamma_2)_{(r)}) \otimes 
S^*(\oplus_{\ell=0}^{r-1} (\gamma_2/\gamma_3)^{\#(r)}[2p^{r-\ell-1}]).
\end{equation}
\end{defn}
We view
$\ol S^*((U_J/\Gamma_3)_{(r)})$ as 
$$\ol S^*((U_J/\Gamma_3)_{(r)}) \ \simeq \ S^*((U_J/\Gamma_3)_{(r)}) /I_3$$
where $I_3 \subset S^*((U_J/\Gamma_3)_{(r)})$ is the ideal generated by 
the relations (\ref{reln:S2}).

\begin{prop}
\label{prop:oleta3}
The map $\eta_{U_J/\Gamma_3,r}$ of Proposition \ref{prop:eta3} factors through the quotient
$S^*((U_J/\Gamma_3)_{(r)}) \twoheadrightarrow \ol S^*((U_J/\Gamma_3)_{(r)})$, thereby determining 
the map of $k$-algebras
\begin{equation}
\label{map:olS*}
\ol \eta_{U_J/\Gamma_3,r} = \tilde f_Q \odot \tilde g: \ol S^*((U_J/\Gamma_3)_{(r)}) \ \to \ H^\bu((U_J/\Gamma_3)_{(r)},k),
\end{equation}
where $\tilde f_Q$ factors the map $\tilde f$ of (\ref{tildef}) via the natural surjection 
$S^*(\oplus_{\ell=0}^{r-1} (\fu_J/\gamma_2)^{\#(\ell+1)}[2])  \to Q((U_J/\Gamma_2)_{(r)})$.

The map $\ol \eta_{U_J/\Gamma_3,r}$ is a map of filtered algebras for the LHS filtration.
Moreover, 
$$gr\{\ol \eta_{U_J/\Gamma_3,r} \} \ = \ f_Q \odot g,$$
where $f_Q: Q((U_J/\Gamma_2)_{(r)}) \to E_\infty^{*,0}$ composed with the natural map $E_\infty^{*,0} \to
H^*((U_J/\Gamma_3)_{(r)},k)$ equals $\tilde f_Q$.
\end{prop}

\begin{proof}
Because $\eta_{U_J,3}(I_3) = 0$, $\eta_{U_J,3}$ induces $\ol \eta_{U_J,3}$.  Since $I_3$ is generated by 
elements of $J_2$, we conclude that $\ol \eta_{U_J,3}$ is given as indicated in (\ref{map:olS*}).
Observe that $Q^*((U_J/\Gamma_2)_{(r)}) = S^*(\oplus_{\ell=0}^{r-1} (\fu_J/\gamma_2)^{\#(\ell+1)}[2])/J_2$ 
inherits an LHS filtration  because $J_2$ is a filtered ideal.

Thus, to prove that $\ol \eta_{U_J/\Gamma_3,r}$ is filtration preserving, it suffices to observe that (1) $ \ol \eta_{U_J/\Gamma_3,r}$ 
restricted to $Q^*((U_J/\Gamma_2)_{(r)})$  is a filtered map of algebras for the LHS filtration 
because it is induced by the inflation map for $(U_J/\Gamma_3)_{(r)}\to (U_J/\Gamma_2)_{(r)}$.
(2) $\eta_{U_J/\Gamma_3,r}$  restricted to $S^*(\oplus_{\ell=0}^{r-1} (\gamma_2/\gamma_3)^{\#(\ell+1)}[2])$ 
is also a filtered map of algebras for the LHS filtration, as shown in the proof of 
Propositions \ref{prop:eta3}.

The identification of $gr\{\ol \eta_{U_J/\Gamma_3,r} \}$ follows from Proposition \ref{prop:eta3} and the 
fact that the LHS filtration on $\ol S^*((U_3)_{(r)})$ is induced by taking the quotient of the filtration for
$S^*((U_J/\Gamma_3)_{(r)})$ by an ideal which is of the form 
$J_2 \otimes S^*(\oplus_{\ell=0}^{r-1} (\gamma_2/\gamma_3)^{\#(r)}[2p^{r-\ell-1}])$.    
\end{proof}

The Andersen-Jantzen spectral sequence of Proposition \ref{prop:AJ}  admits a natural action of $T$
whose $T$-weights are identified using (\ref{summands}).
We envision that this spectral sequence should
enable the extension to $U_J/\Gamma_{v+1}$ of our considerations of
$U_J/\Gamma_3$.  

For $U_J/\Gamma_3$, we verify below that the elements of $S^*((U_J/\Gamma_3)_{(r)})$ are permanent 
cycles of the AJ spectral sequence.  It would be of interest to know whether or not all permanent cycles
in $S^*(\oplus_{\ell=0}^{r-1} (\fu_J/\gamma_{3})^{\#(\ell+1)}[2])$ belong to $S^*((U_J/\Gamma_3)_{(r)})$.

\begin{prop}
\label{prop:permanent3-AJ}
Retain the hypotheses and notation of Proposition \ref{prop:ABS}.  Every element 
$$z \in S^*((U_J/\Gamma_{3})_{(r)}) \subset S^*(\oplus_{\ell=0}^{r-1} (\fu_J/\gamma_{3})^{\#(\ell+1)}[2])  
\subset {}^{AJ}E_1^{*,*}((U_J/\Gamma_{3})_{(r)})$$
is a permanent cycle for the AJ spectral sequence. 

Moreover, the map \ $\eta_{U_J/\Gamma_3,r}: S^*((U_J/\Gamma_3)_{(r)}) \ \to \ H^\bu((U_J/\Gamma_3)_{(r)},k)$ \
sends $z \in S^*((U_J/\Gamma_3)_{(r)})$ to the cohomology class $\eta_{U_J/\Gamma_3,r}(z)$ 
represented at the $E_1^{**}$-page of the AJ spectral sequence by the image 
of $z$ in ${}^{AJ}gr(S^*((U_J/\Gamma_3)_{(r)}) \simeq S^*((U_J/\Gamma_3)_{(r)}) \subset {}^{AJ}E_1^{*,*}((U_J/\Gamma_3)_{(r)})$.
\end{prop}

\begin{proof}

The map $\eta_{U_J/\Gamma_2,r}: S^*((U_J/\Gamma_2)_{(r)}) \ \to \ H^\bu((U_J/\Gamma_2)_{(r)},k)$ of Definition 
\ref{defn:eta} 
is the tensor power of natural embeddings $S^*(H^2(\bG_{a(r)},k)) \to H^\bu(\bG_{a(r)},k)$ as in (\ref{Ga-coh})
and so can be identified with ${}^{AJ}gr(\eta_{U_J/\Gamma_2,r})$.  

Recall that $S^*((U_J/\Gamma_{3})_{(r)})$ is the tensor
product of its restrictions to \\
$S^*(\oplus_{\ell=0}^{r-1} (\fu_J/\gamma_2)^{\#(\ell+1)}[2])$ and
$S^*(\oplus_{\ell=0}^{r-1} (\gamma_2/\gamma_{3})^{\#(r)}[2p^{r-\ell-1}])$.
Using functoriality for the map $U_J/\Gamma_{3} \to U_J/\Gamma_2$
and multiplicativity of $\eta_{U/\Gamma_{3},r}$, we invoke a simple induction argument to conclude that it
suffices to prove the assertions for $z \in S^*(\oplus_{\ell=0}^{r-1} (\gamma_2/\gamma_{3})^{\#(r)}[2p^{r-\ell-1}])$.

Let $\beta$ be a root of $U_J$ of level $2$ and $0 \leq \ell < r$.  Consider $\eta_{U/\Gamma_{3},r-\ell}$ applied to
$(X^\beta(0))$ (corresponding to $(x_\beta^{(0)})^{p^{r-1}} \in E_1^{*,*}$ as in Notation \ref{note:compare}), giving an element
of $H^{2p^{r-\ell-1}}((U_J/\Gamma_{3})_{(r-\ell)},k)$ of weight $p^{r-\ell-1}\cdot \beta$.
An inspection of (\ref{summands}) verifies that the 
only possible representative in ${}^{AJ}E_1^{*,*}((U_J/\Gamma_{3})_{(r-\ell)})$
of this weight and degree is $(x_\beta^{(0)})^{p^{r-1-\ell}}$.  

The following square commutes
\begin{equation}
\label{commute:StoSbar}
\begin{xy}*!C\xybox{%
\xymatrix{
S^*((U_J/\Gamma_{3})_{(r-\ell)}) \ar[d]_-{\eta_{U/\Gamma_{3},r-\ell}} \ar[r]^-{F^{\ell*}} & 
\ol S^*((U_J/\Gamma_{3})_{(r)}) \ar[d]^-{\eta_{U/\Gamma_{3},r}} \\
H^*((U_J/\Gamma_{3})_{(r-\ell)},k)  \ar[r]^-{F^{\ell*}}  & H^*((U_J/\Gamma_{3})_{(r)},k)
}
}\end{xy}
\end{equation}
because the maps $\eta$ are defined over $\bF_p$.  Moreover, $(F^{\ell})^*$ determines a map of 
AJ spectral sequences, so that the representative in ${}^{AJ}E_1^{*,*}((U_J/\Gamma_3)_{(r-\ell)})$ of \\
$\eta_{U_J/\Gamma_{3},r-\ell}(X^\beta(0))$ is sent to the representative 
in ${}^{AJ}E_1^{*,*}((U_J/\Gamma_{3})_{(r)})$ of \\
$\eta_{U_J/\Gamma_{3},r}(X^\beta(\ell))$.  In particular, $(x_\beta^{(\ell)})^{p^{r-1-\ell}} \in
S^*((U_J/\Gamma_{3})_{(r)})$ is a permanent cycle in the AJ spectral sequence representing
$\eta_{U/\Gamma_{3},r}(X^\beta(\ell))$ for each $\beta$  and each $0 \leq \ell < r$.
Now, using mulitipicativity, we conclude the assertions of the proposition for $z \in 
S^*(\oplus_{\ell=0}^{r-1} (\gamma_2/\gamma_{3})^{\#(r)}[2p^{r-\ell-1}])$ and thus the proposition as
stated.
\end{proof}

As a corollary of Proposition \ref{prop:permanent3-AJ}, we conclude the following description 
$gr\{ \eta_{U_J/\Gamma_3,r} \}$.  This is of some interest for the AJ-filtration is intrinsic to $U_J/\Gamma_3$.

\begin{cor}
\label{cor:AJfiltGamma3}
As for $gr\{ \eta_{U_J/\gamma_3,r} \}$ in Proposition \ref{prop:eta3}, 
$${}^{AJ}gr\{ \eta_{U_J/\gamma_3,r} \} = f \odot g:  S^*((U_J/\Gamma_3)_{(r)}) \ \to \ {}^{AJ}gr\{ H^*((U_J/\Gamma_3)_{(r)},k)\}.$$
\end{cor}

\vskip .2in

%%%%%%%%%%%%%%%%%%%%%%%%%%%%%
%%%%%%%%%%%%%%%%%%%%%%%%%%%%%%

\section{The map  $ \ol \eta_{U_3,r}: \ol S^*((U_3)_{(r)}) \to H^\bu((U_3)_{(r)},k)$}
\label{sec:U3}

We apply the results of the previous section to the Heisenberg group $U_3$ and its Frobenius kernels $(U_3)_{(r)}$.

\begin{ex}
\label{ex:Sbar-U3}
$\ol S^*((U_3)_{(r)})$ is generated by
elements $X^{1,3}(\ell) =  (x_{1,3}^{(\ell)})^{p^r-\ell-1}\in (\gamma_2)^{\#(r)}[2p^{r-\ell-1}]$ 
and $(X^{1,2}(\ell))^{p^{-r+\ell+1}} =  x_{1,2}^{(\ell)}, \ (X^{2,3}(\ell))^{p^{-r+\ell+1}} =
x_{2,3}^{(\ell)}\in (\fu_3/\gamma_2)^{\#(\ell+1)}[2]$
with $0 \leq \ell < r$.  A set of relations for $\ol S^*((U_3)_{(r)})$  is  given by the
special case of (\ref{reln:S2}):
\begin{equation}
\label{reln:S2-U3}
(x_{1,2}^{(\ell)})^{p^{j+1}} \otimes 
x_{2,3}^{(\ell+1+j)} - (x_{2,3}^{(\ell)})^{p^{j+1}} \otimes x_{1,3}^{(\ell+1+j)}
\ = \ 0 \ \in H^{2p^{j+1}+2}((U_3)_{(r)},k)
\end{equation}
for each $\ell, j \geq 0$ such that $\ell+ 1 +j < r$.

Generators and relations for  $Q((U_3/\Gamma_2)_{(r)})$ are obtained from those for $\ol S^*((U_3)_{(r)})$ 
by setting each $X^{1,3}(\ell)$ equal to 0.
\end{ex}

Composition with the quotient map $U_J/\Gamma_3 \to U_J/\Gamma_2$ determines $V_r(U_J/\Gamma_3) \to V_r(U_J/\Gamma_2)$
and thus the map of $k$-algebras $k[V_r(U_J/\Gamma_2)] \to k[V_r(U_J/\Gamma_3)]$.  We denote the image of this map
by $\ol{k[V_r(U_J/\Gamma_3)]} \subset k[V_r(U_J/\Gamma_3)]$. 

In the proof of the following proposition we use the description of $k[V_r(U_J/\Gamma_3)]$ in terms of generators and relations 
which follows immediately from Proposition \ref{prop:def-relns} and the explicit description of $k[V_r(GL_N)]$ given in
Theorem \ref{thm:SFB-factor}.  Namely, $k[V_r(U_J/\Gamma_3)]$ is generated by $X^{i,j}(\ell)$ where $1 \leq i \leq j \leq 3$
and $0 \leq \ell < r$; a set of relations is given by 
\begin{equation}
\label{reln:VrUGamma3}
X^{1,2}(\ell)\cdot X^{2,3}(\ell^\prime) - X^{2,3}(\ell)\cdot X^{1,2}(\ell^\prime), \quad 0 \leq \ell < \ell^\prime.
\end{equation}

\begin{prop}
\label{prop:ratU3}
The  coordinate algebra $k[V_r(U_J/\Gamma_3)]$  admits a natural tensor product 
decomposition as $k$-algebras, 
$$k[V_r(U_J/\Gamma_3)] \ = \  \ol{k[V_r(U_3/\Gamma_2)]} \otimes S^*(\oplus_{\ell=0}^{r-1} (\gamma_2/\gamma_3)^{\#(r)}[2p^{r-\ell-1}]).$$

In the special case $U_J = U_3$,  $\ol {k[V_r(U_3/\Gamma_2)]}$ is an integral domain smooth outside of the origin
with field of fractions a purely transcendental extension of  transcendence degree $r+1$.

Consequently, $k[V_r(U_3)]$ is an  integrally closed domain of dimension $2r+1$.
\end{prop}

\begin{proof}
We identify $k[V_r(U_J/\Gamma_2)]$ with 
\begin{equation}
\label{identify:UJGamma2}
S^*(\oplus_{\ell=0}^{r-1} (\fu_J/\gamma_2)^{\#(r)}[2p^{r-\ell-1}]) \ \simeq \ k[X^{1,2}(\ell),X^{2,3}(\ell^\prime);\ 0 \leq \ell, \ell^\prime < r]
\end{equation}
and $\ol{k[V_r(U_3/\Gamma_2)]}$ with the quotient of $k[V_r(U_J/\Gamma_2)]$  by the relations (\ref{reln:VrUGamma3}).
The tensor product decomposition is immediate from the observation that these relations 
do not involve elements of $S^*(\oplus_{\ell=0}^{r-1} (\gamma_2/\gamma_3)^{\#(r)}[2p^{r-\ell-1}]).$

If $r=1$, then $V_r(U_3) = \fu_3$ so that $\ol{k[V_1(U_3/\Gamma_2)]} \ \simeq \ k[X^{1,3}(0),X^{2,3}(0)]$.   
For the remainder of the proof, we assume $r > 1$.

For any $\ell_1, 0 \leq \ell_1 < r$, the algebra $\ol{k[V_r(U_3/\Gamma_2)]}[(X^{2,3}(\ell_1))^{-1}]$ is isomorphic to 
$$k[X^{2,3}(\ell), 0 \leq \ell < r; X^{1,2}(\ell_1)][(X^{2,3}(\ell_1))^{-1}],$$
since $X^{1,2}(\ell) = X^{1,2}(\ell_1)(X^{2,3}(\ell_1))^{-1}X^{2,3}(\ell)$;
similarly, for any $\ell_0, 0 \leq \ell_0 < r$,  the algebra $\ol{k[V_r(U_3/\Gamma_2)]}[(X^{1,2}(\ell_0))^{-1}]$ is isomorphic to 
$$k[X^{1,2}(\ell), 0 \leq \ell < r; X^{2,3}(\ell_0)][(X^{1,2}(\ell_0))^{-1}].$$
This verifies the computation of the field of fractions of $\ol{k[V_r(U_3/\Gamma_2)]}$ and shows that $\ol{k[V_r(U_3/\Gamma_2)]}$
is smooth outside the common zeros of $\{ X^{1,2}(\ell_0), X^{2,3}(\ell_1); 0 \leq \ell_0,\ell_1 < r \}$;
namely, the origin.
A theorem of Serre (see \cite[Thm 39]{M}) tells us that  $\ol{k[V_r(U_3/\Gamma_2)]}$ is an integrally 
closed domain since the codimension of this zero locus is at least 2.
\end{proof}

\begin{prop}
\label{prop:VtoSbar}
There is a naturally constructed injective map   
\begin{equation}
\label{map:theta}\
\theta_{U_J/\Gamma_3,r}: k[V_r(U_J/\Gamma_3)] \ \to \ \ol S^*((U_J/\Gamma_3)_{(r)})
\end{equation} 
such that $\eta_{U_J/\Gamma_3,r}: S^*((U_J/\Gamma_3)_{(r)} \to H^\bu((U_J/\Gamma_3)_{(r)},k)$
as defined in Definition \ref{defn:eta}  factors through $ \ol \eta_{U_J/\Gamma_3,r} \circ \theta_{U_J/\Gamma_3,r}:
k[V_r(U_J/\Gamma_3)]  \to H^\bu((U_J/\Gamma_3)_{(r)},k),$
where $\ol \eta_{U_J/\Gamma_3,r}$ is given in Proposition \ref{prop:oleta3}.

In the special case $U_J = U_3$, 
$$\ol \eta_{U_3,r} \circ \theta_{U_3,r} = \ol \phi_{U_3,r}: k[V_r(U_3)] \to H^\bu((U_3)_{(r)},k),$$
where $\ol \phi_{U_3,r}$ is given in Proposition \ref{prop:coh-UN}.
\end{prop}

\begin{proof}
Consider the defining quotient map $q_{U_J/\Gamma3,r}: S^*(\oplus_{\ell = 0}^{r-1} (\fu_J/\gamma_3)^{\#(r)}[2p^{r-\ell-1}])
\to k[V_r(U_J/\Gamma_3)]$.   We see by inspection that the kernel of $q_{U_J/\Gamma3,r}$
is generated by the intersection of $J_2 \subset 
S^*(\oplus_{\ell = 0}^{r-1} (\fu_J/\gamma_2)^{\#(r)}[2p^{r-\ell-1}])$ with $S^*(\oplus_{\ell = 0}^{r-1} (\fu_J/\gamma_2)^{\#(r)}[2p^{r-\ell-1}])$.
This determines $\theta_{U_J/\Gamma_3,r}$ fitting in the commutative square 
\begin{equation}
\label{commute:theta}
\begin{xy}*!C\xybox{%
\xymatrix{
S^*(\oplus_{\ell = 0}^{r-1} (\fu_J/\gamma_3)^{\#(r)}[2p^{r-\ell-1}]) \ar[r]^-{q_{U_J/\Gamma_3,r}} \ar[d] &  k[V_r(U_J/\Gamma_3)]  
\ar[d]^{\theta_{U_J/\Gamma_3,r}} \\
S^*((U_J/\Gamma_3)_{(r)}) \ar[r] & \ol S^*((U_J/\Gamma_3)_{(r)}). \\
 }
}\end{xy}
\end{equation}

The fact that the kernel of $q_{U_J/\Gamma_3,r}$ is the intesection of $I_3$ with 
$S^*(\oplus_{\ell = 0}^{r-1} (\fu_J/\gamma_3)^{\#(r)}[2p^{r-\ell-1}])$ (also generated by the intersection of $J_2 \subset 
S^*(\oplus_{\ell = 0}^{r-1} (\fu_J/\gamma_2)^{\#(r)}[2p^{r-\ell-1}])$)
implies the injectivity of $\theta_{U_J/\Gamma_3,r}$.  

The equality $\eta_{U_3,r} = \phi_{U_3,r}$ of Proposition \ref{prop:1-dim} together with
 the surjectivity of $q_{U_J/\Gamma_3,r}$ implies the 
equality $\ol \eta_{U_3,r} \circ \theta_{U_3,r} \ = \ \ol \phi_{U_3,r}$.  
\end{proof}

We next observe that the tensor product decomposition of Proposition \ref{prop:ratU3} is 
respected by the map $\theta_{U_J/\Gamma_3}$.  This enables us to show in the following proposition 
that $\theta_{U_3,r}: k[V_r(U_3)] \to \ol S^*((U_3)_{(r)})$ is a finite map of integral domains.

\begin{prop}
\label{prop:theta-tensor}
The map $\theta_{U_3,r}$  of Proposition \ref{prop:VtoSbar} can be written as the tensor product
$$\ol \theta_{U_3,r} \otimes 1:  \ol{k[V_r(U_J/\Gamma_3)]}  \otimes S^*((\gamma_2/\gamma_3)^{\#(r)}[2p^{r-\ell-1}]) \ \to $$
$$ \to \ Q^*((U_3/\Gamma_2)_{(r)}) \otimes S^*((\gamma_2)^{\#(r)}[2p^{r-\ell-1}]) \ = \  \ol S^*((U_3)_{(r)}).$$

The map $\ol \theta_{U_3,r}: \ol{k[V_r(U_3)]}   \to Q^*((U_3/\Gamma_2)_{(r)})$ is a finite map of integral domains
 of degree $p^{\frac{(r+2)(r-1)}{2}}$ obtained by taking $p^{r-\ell-1}$-st roots of $X^{1,2}(\ell), X^{2,3}(\ell)$ 
 for each $\ell, 0 \leq \ell < r$.   
 Thus, $\theta_{U_3,r}: k[V_r(U_3)] \to \ol S^*((U_3)_{(r)})$ is a finite map of integral domains.
 \end{prop}

\begin{proof}
The fact that $\theta_{U_J/\Gamma_3,r}$
is a tensor product of the form $\ol \theta_{U_3,r} \otimes 1$ arises from the fact that the tensor decomposition
of $k[V_r(U_3)]$ in Proposition \ref{prop:ratU3} and that of $\ol S^*((U_3)_{(r)})$ in (\ref{Sbar-split}) both arise 
because the relations do not involve weights of level 2.   The fact that $\theta_{U_J/\Gamma_3,r}$ is essentially
the identity on $S^*((\gamma_2/\gamma_3)^{\#(r)}[2p^{r-\ell-1})$ can be traced back to the definition of 
$X^{i,j}(\ell) \in \gl_N^{\#(r)}[2p^{r-\ell-1}]$ prior to Theorem \ref{thm:SFB-factor}.

The $k$-algebra $k[V_r(U_3)]$ is an integral domain by Proposition \ref{prop:ratU3}.
Arguing as in the proof of Proposition \ref{prop:ratU3}, we verify that $Q^*((U_3/\Gamma_2)_{(r)})[(x_\alpha^{(0)})^{-1}]$
is the localization of the  polynomial algebra on generators $x_{\alpha^\prime}^{(\ell)}, 0 \leq \ell < r; Y^{1,2}(0)$ with
$x_\alpha^{(0)}$ inverted.  Thus, to show that $Q^*((U_3/\Gamma_2)_{(r)})$ is a domain it suffices 
to show that the localization map
 $Q^*((U_3/\Gamma_2)_{(r)}) \to Q^*((U_3/\Gamma_2)_{(r)})[(x_{\alpha}^{(0})^{-1}]$
is injective.  This is verified by examining the relations (\ref{reln:S2}) to show that $x_\alpha^{(0)} 
\in Q^*((U_3/\Gamma_2)_{(r)})$ is not a zero-divisor.

Because $k[V_r(U_3)]$ is a domain, $F^r = \psi_{U_3,r} \circ \ol \eta_{U,3} \circ \theta_{U_3,r}: k[V_r(U_3)] \ \to \ k[V_r(U_3)]$
is injective and thus $\theta_{U_3,r}$ is injective.   Since $S^*(\oplus_{\ell=0}^{r-1} (\fu_3)^{\#(r)}[2p^{r-\ell-1}]) \to 
S^*((U_3)_{(r)})$ is obtaining by taking $p^{r-\ell-1}$-st roots of 
$(x_\alpha^{(\ell)})^{p^{r-\ell-1}}, (x_{\alpha^\prime}^{(\ell)})^{p^{r-\ell-1}} $ for each $\ell, 0 \leq \ell < r$,
we conclude that $ Q^*((U_3/\Gamma_2)_{(r)})$ is similarly obtained from $\ol{k[V_r(U_3/\Gamma_2)]}$.

To compute the degree of $\ol \theta_{U_3,r}$,
we consider the map $\ol{k[V_r(U_3/\Gamma_2)]}[(X^{1,2}(0))^{-1}] \to Q^*((U_3/\Gamma_2)_{(r)})[(x_\alpha^{(0)})^{-1}]$ and utilize
the facts that $x_{\alpha^\prime}^{(\ell)}$ is the $p^{r-\ell-1}$-th root of the image of $X^{2,3}(\ell)$ and that 
$x_\alpha^{(0)}$ is the $p^{r-1}$-st root of the image of $X^{1,2}(0)$ (using the notation of Proposition \ref{prop:ratU3}).

Finally, since $ \theta_{U_3,r} \ = \ \ol\theta_{U_3,r}\otimes1$, the fact that  $\ol \theta_{U_3,r}$ is a finite map of 
integral domains implies that $ \theta_{U_3,r}$ is also a finite map of integral domains 
\end{proof}

The following theorem summarizes what we know about $\ol \eta_{U_3,r}$.

\begin{thm}
\label{thm:U3r}
Retain the hypotheses and notation of Proposition \ref{prop:ABS}.
\begin{enumerate}
\item
$\ol \eta_{U_3,r} \ = \ \phi_{U_3,r}.$
\item
$\ol \eta_{U_3,r}: \ol S^*((U_3)_{(r)}) \ \to \ H^\bu((U_3)_{(r)},k)$ is injective.  
\item
$\ol \eta_{U_3,r}$ is surjective onto $p^r$-th powers of elements of $H^\bu((U_3)_{(r)},k)$.
\item
$gr(\ol \eta_{U_3,r})$ factors through
$$f_Q \otimes g: \ol S^*((U_3)_{(r)})  \simeq Q^*((U_3/\Gamma_2)_{(r)}) \otimes 
S^*(\oplus_{\ell = 0}^{r-1} (\gamma_2/\gamma_3)^{\#(r)}[2p^{r-\ell-1}])
\ \to \ E_\infty^{*,0} \otimes E_\infty^{0,*}.$$ 
\item
$gr(\ol \eta_{U_3,r}): \ol S^*((U_3)_{(r)}) \to \ gr\{H^\bu((U_3)_{(r)},k)$ is injective.
\item
$gr(\ol \eta_{U_J/\Gamma_3,r})$ is surjective onto $p$-th powers of elements of $gr\{H^\bu((U_3)_{(r)},k)$.
\end{enumerate}
\end{thm}

\begin{proof}
The equality $\ol \eta_{r} \ = \ (\phi_{U_3,r})_|$ of  (\ref{diagram:compat-eta}) implies that
$ \ol \eta_{U_3,r} \ = \ \phi_{U_3,r}$, since  $\eta_{U_3,r}, \ \phi_{U_3,r}$ are determined by
the Frobenius twists of the basic maps $\ol \eta_{r},\ (\phi_{U_3,r})_|$.

By Proposition \ref{prop:coh-UN}, $F^r = \psi_{U_3,r} \circ \ol \phi_{U_3,r}: k[V_r(U_3)] \to k[V_r(U_3)]$.
Since $k[V_r(U_3)]$ is a domain by Proposition \ref{prop:ratU3}, $F^r$ and thus also $\ol \phi_{U_3,r}$ are injective.  
Since the $p^{r-1}$-st power of each element in $S^*((U_3)_{(r)})$ lies in the 
image of $S^*(\oplus_{\ell = 0}^{r-1} (\fu_3)^{\#(r)}[2p^{r-\ell-1}])$, the commutativity of (\ref{commute:theta})
and the surjectivity of $S^*((U_J/\Gamma_3)_{(r)})  \to \ol S^*((U_J/\Gamma_3)_{(r)})$
imply that the $p^{r-1}$-st power of an element in $\ol S^*((U_3)_{(r)})$ lies in  the image of $\theta_{U_3,r}$. 
Thus, any element in the kernel of $\ol \eta_{U_3,r}$ must have $p^{r-1}$-st power which is in the kernel
of $\ol \phi_{U_3,r}$ which we have observed is trivial.  Since $\ol S^*((U_3)_{(r)})$ is a domain by 
Proposition \ref{prop:theta-tensor}, we conclude that $\ol \eta_{U_3,r}$ must be injective.

The surjectivity statement of (3) follows from the fact that the composition $F^r = \psi_{U_3,r} \circ \ol \phi_{U_3,r}$ is surjective
onto $p^r$-th powers.

The factorization of $f\odot g$ through $f\otimes g: S^*((U_3)_{(r)})  \otimes 
S^*(\oplus_{\ell = 0}^{r-1} (\gamma_2/\gamma_3)^{\#(r)}[2p^{r-\ell-1}]) \\
\ \to \ E_\infty^{*,0} \otimes E_\infty^{0,*}$ is given by the proof of  Proposition \ref{prop:eta3}.  This is 
easily seen to determine the factorization of $gr(\ol \eta_{U_3,r})$ as asserted in (4).

Granted the injectivity (1), to prove the injectivity of $gr(\ol \eta_{U_3,r})$ 
as asserted in (5) it suffices to show that both $\tilde f_Q$ and $\tilde g$ are
filtration level preserving.   Since the filtration level on $Q((U_J/\Gamma_2)_{(r)}$ equals the cohomological degree
and $\tilde f_Q$ can only increase filtration level (since it is a map of filtered algebras), $\tilde f_Q$ preserves 
filtration level.  Since the composition of $\tilde g: S^*(\oplus_{\ell=0}^{r-1}(\gamma_2/\gamma_3)^{\#(r)}[2p^{r-\ell-1}]) \to 
H^\bu((U_3)_{(r)},k)$ with the natural map $E^{0,*}_\infty$ is injective, we conclude that $\tilde g$ sends every
non-zero element to a cohomology class of filtration level 0 and therefore must also preserve filtration level.

Finally, to prove that $gr(\ol \eta_{U_3,r})$ is surjective onto $p$-th powers, it suffices to prove that the $p$-th power
of any permanent cycle, $z \in Z_2^{*,*} \subset E_2^{*,*}$, lies in the image of $S^*((U_3)_{(r)})$.  Write $z = x +y$
with $x$ defined as the sum of those summands of $z$ lying in $S^*(\oplus_{\ell=0}^{r-1}(\fu_3)^{\#(\ell+1)}[2])$ and
$y$ as the sum of the remaining summands of $z$, where $z$ is written as a sum of terms using the decomposition
(\ref{UGamma3-tensor}).  Then $y^p = 0 \in E_2^{*,*}$, so that $z^p = x^p \in
Z_2^{*,*} \cap S^*(\oplus_{\ell=0}^{r-1}(\fu_3)^{\#(\ell+1)}[2])$.  By Corollary 2.4, either $z^p$ is a boundary or 
$z^p \in S^*((U_3)_{(r)})$.  This completes the proof of (6).
\end{proof}

Observe that the (adjoint) action of $T$ on $S^*((U_J)_{(r)})$ induces an action on $\ol S^*((U_J)_{(r)}) = S^*((U_J)_{(r)}) /I_3$
which restricts to an action on the tensor factor $Q((U_J/\Gamma_2)_{(r)})$ of $\ol S^*((U_J)_{(r)})$ because 
the relations (\ref{eqn:relnU3}) are generated by $T$-eigenvalues.  The following
proposition is a consequence of the description of $\theta_{U_3,r}$ given in Proposition \ref{prop:theta-tensor}.

\begin{prop}
\label{T3-invar}
With respect to the above action, we have the equality
$$H^0(T_{3(r)},Q((U_3/\Gamma_2)_{(r)}) \ = \ \ol{k[V_r(U_3)]}.$$
Consequently, the restriction of $\ol \eta_{U_3,r}: \ol S^*((U_3)_{(r)}) \to H^\bu((U_3)_{(r)},k)$ to 
$(T_3)_{(r)}$-invariants has the form 
\begin{equation}
\label{eqn:U3iso}
 k[V_r(U_3)] \ \to \ H^\bu((B_3)_{(r)},k).
\end{equation}
\end{prop}

\begin{proof}
By definition of $X^{i,j}(\ell)$ as an element of $\gl_N^{\#(r)}[2p^{r-\ell-1}]$ (see Theorem \ref{thm:SFB-factor}),
we see that the images in  $\ol{k[V_r(U_3)]} \subset Q((U_3/\Gamma_2)_{(r)}$ of $X^{i,j}(\ell)$ for $1 \leq i <  j \leq 3$  
are $T_{3(r)}$ invariant.  On the other hand, using Proposition \ref{prop:theta-tensor} we write $Q((U_3/\Gamma_2)_{(r)}$
as a free module over $\ol{k[V_r(U_3)]}$ generated by powers of the $p^{r-\ell -1}$-st roots of the images of
$X^{i,j}(\ell)$ viewed as images of elements of $S^*(\oplus_{\ell=0}^{r-1} \fu^{\#(\ell+1)}[2])$.    If a power of such
a root is not divisible by $p^{r-\ell -1}$, then it is a non-trivial eigenvector for the semi-simple action of $T_{3(r)}$.
This implies the asserted equality.

Since $\ol S^*((U_3)_{(r)}) \simeq Q((U_3/\Gamma_2)_{(r)}) \otimes 
S^*(\oplus_{\ell=0}^{r-1}(\gamma_2/\gamma_3)^{\#(r)}[2p^{r-\ell-1}])$, we conclude that
$H^0(T_{3(r)},\ol S^*((U_3)_{(r)}))$ equals $\ol{k[V_r(U_3)]} \otimes S^*(\oplus_{\ell=0}^{r-1}(\gamma_2/\gamma_3)^{\#(r)}[2p^{r-\ell-1}])$
which equals $k[V_r(U_3)]$ by Proposition \ref{prop:ratU3}.  The spectral sequence for the extension 
$$1 \to (U_3)_{((r)} \ \to \ (B_3)_{(r)} \ \to \ (T_3)_{(r)}  \to 1$$
and the semi-simplicity of $(T_3)_{(r)}$ imply the equality 
$$H^0((T_3)_{(r)},H^*((U_3)_{(r)},k)) \ = \ H^*((B_3)_{(r)},k).$$
This implies that $H^0(T_{(r)},-)$ applied to $\ol \eta_{U_3,r}$ yields $ k[V_r(U_3)] \ \to \ H^\bu((B_3)_{(r)},k)$.
\end{proof}

\vskip .2in

%%%%%%%%%%%%%%Section1%%%%%%%%%%%%
%%%%%%%%%%%%%%%%%%%%%%%%%%%%%%%

\section{Questions}
\label{sec:questions}

Here are a few of the many questions encountered, but not answered, in this paper.

\begin{question}
\label{ques:unknown}  
 For $V$ as in Theorem \ref{thm:inj}, does there exist a non-nilpotent cohomology class
$\alpha \in H^*(V,k)$ each of whose restrictions $i_r^*(\alpha) \in H^*(V_{(r)},k)$ is nilpotent?
\end{question}

\vskip.1in
\begin{question}
Under what conditions on the unipotent group $V$ is the image of the restriction map   $H^*(V,k) \to H^*(V_{(r)},k)$ 
finitely generated?
\end{question}

\begin{question}
Are there natural Steenrod operations on the Andersen-Jantzen spectral sequence which satisfy the usual relationship
with respect to differentials including the Kudo transgression theorem?
\end{question}

\begin{question}
Can one compare the LHS-filtration and the AJ-filtration on the Hochschild complex $C^*((U_J/\Gamma_3)_{(r)})$?
\end{question} 

\begin{question}
Is the map $\ol \phi_{U_3,r}: k[V_r(U_3)] \to (H^*((U_3)_{(r)},k))_{red}$ an isomorphism?
\end{question}

\begin{question}
Under what conditions on the unipotent algebraic group $U$ is the $k$-algebra $k[V_r(U)]$
reduced (i.e., has no non-trivial nilpotent elements) for all $r > 0$?
\end{question}

\begin{question}
Can we establish the unipotent analogue of the ``matrix $p$-th power relation"
$$ S_{i,j,\ell} \quad = \quad \sum_{t_1,\ldots ,t_{p-1}} X^{i,t_1}(\ell) \cdot X^{t_1,t_2}(\ell) \cdots X^{t_{p-1},j}(\ell)$$
using the AJ spectral sequence.
\end{question}

%%%%%%%%%%%%%%%%%%%%%%%%%%%%%%%%%%%%%%%%%%%
%%%%%%%%%%%%%%%%%%%%%%%%%%%%%%%%%%%%%%%%%%%

\end{document}